\documentclass[11pt]{article}
\usepackage[utf8]{inputenc}
\usepackage{cite}

\title{Computing $\nu$-invariants of Joyce's compact $G_2$-manifolds }
\author{Christopher Scaduto \vspace{.1cm}  \\ {\emph{University of Miami, Coral Gables, FL}}\\ \texttt{c.scaduto@math.miami.edu}} 

\date{}

\usepackage[a4paper, total={6.1in, 8.5in}]{geometry}
\usepackage{graphicx}

\usepackage{times}

\usepackage{amssymb}
\usepackage{tikz-cd}
\usepackage{titling}
\usepackage{mathrsfs} 
\usepackage{booktabs}
\usepackage[all]{xy}
\usepackage{amsthm}
\usepackage{diagbox}
\usepackage{tabularx}
\usepackage{amscd}
\usepackage{caption}
\usepackage{nicefrac}
\usepackage{amsmath}
\usepackage{mathtools}
\usepackage{tikz}
\usepackage{mathabx}
\usepackage{dsfont}
\usepackage{lipsum}
\usepackage{mwe}
\usepackage{slashed}
\usepackage{rotating}
\usepackage{subcaption}
\usepackage[colorlinks,pagebackref,hypertexnames=false]{hyperref} \usepackage[alphabetic,backrefs,msc-links]{amsrefs}
\usepackage{epstopdf,pinlabel}
\definecolor{mint}{HTML}{239B56}
\hypersetup{
    colorlinks=true,
    linkcolor=blue,
    filecolor=magenta,      
    urlcolor=cyan,
    citecolor=mint
}
\usepackage{pgfplots}
\pgfplotsset{compat=newest}
\usetikzlibrary{shapes.geometric}

\newcolumntype{Y}{>{\centering\arraybackslash}X}

\usepackage{sectsty}

\sectionfont{\fontsize{12}{15}\selectfont}

\newcommand{\R}{\mathbb{R}}
\newcommand{\C}{\mathbb{C}}
\newcommand{\Z}{\mathbb{Z}}

\newcommand{\cO}{\mathcal{O}}

\newcommand{\nodea}{\begin{tikzpicture}
\node at (0,0) [circle,fill=black,scale=0.7] () {};
\end{tikzpicture}}

\newcommand{\nodeb}{\begin{tikzpicture}
\node at (0,0) [circle,draw,scale=0.65] () {};
\end{tikzpicture}}

\newcommand{\nodec}{\begin{tikzpicture}
\node at (0,0) [circle,draw,fill=green,scale=0.7] () {};
\end{tikzpicture}}

\newcommand{\noded}{\begin{tikzpicture}[square/.style={regular polygon,regular polygon sides=3}]
\node at (0,0) [square,draw,fill=cyan,scale=0.45] () {};
\end{tikzpicture}}

\newcommand{\nodee}{\begin{tikzpicture}[triangle/.style={regular polygon,regular polygon sides=4}]
\node at (0,0) [triangle,draw,fill=yellow,scale=0.7] () {};
\end{tikzpicture}}

\newcommand{\nodef}{\begin{tikzpicture}[pentagon/.style={regular polygon,regular polygon sides=5}]
\node at (0,0) [pentagon,draw,fill=red,scale=0.7] () {};
\end{tikzpicture}}

\newtheorem{theorem}{Theorem}[section]
\newtheorem{prop}[theorem]{Proposition}

\newtheorem{corollary}[theorem]{Corollary}
\newtheorem{remark}[theorem]{Remark}

\newtheorem{hypothesis}[theorem]{Hypothesis}
\newtheorem{situation}[theorem]{Situation}

\begin{document}

\maketitle

\vspace{-.4cm}
\begin{abstract}
    Crowley and Nordstr\"{o}m introduced an invariant of $G_2$-structures on the tangent bundle of a closed 7-manifold, taking values in the integers modulo 48. Using the spectral description of this invariant due to Crowley, Goette and Nordstr\"{o}m, we compute it for many of the closed torsion-free $G_2$-manifolds defined by Joyce's generalized Kummer construction.
\end{abstract}

\vspace{.2cm}

\section{Introduction}
In \cite{cn-newinvs}, Crowley and Nordstr\"{o}m introduced an invariant $\nu(M,\phi)\in\Z/48$ for a closed 7-manifold $M$ equipped with a $G_2$-structure $\phi$ on its tangent bundle, invariant under homotopies of $\phi$. In the case that $M$ has a metric $g$ with holonomy contained in $G_2$ and associated $G_2$-structure $\phi_g$, Crowley, Goette and Nordstr\"{o}m \cite{cgn} showed that this invariant has the spectral description
\[
    \nu(M,\phi_g) \equiv 3\eta(B_M) - 24\eta(D_M) + 24(1+b_1(M)) \mod 48
\]
Here $\eta(B_M)$ is the $\eta$-invariant of the odd signature operator $B_M$ for the Riemannian manifold $(M,g)$, and $D_M$ is the associated spin Dirac operator. In fact, the authors show that
\[
    \bar{\nu}(M,g) := 3\eta(B_M) - 24\eta(D_M) \in \Z
\]
is an invariant of the torsion-free $G_2$-manifold $(M,g)$ which is locally constant on the moduli space of metrics on $M$ with holonomy contained in $G_2$.

There are only a handful of methods available to construct closed $G_2$-manifolds. The first is the generalized Kummer construction of Joyce \cite{joyce-g2, joyce-book}. There is also the twisted connected sum method of Kovalev \cite{kovalev}, generalized by Corti--Haskins--Nordstr\"{o}m--Pacini \cite{chnp,chnp-weak}. A further generalization, that of ``extra-twisted'' connected sums, is considered by Crowley--Goette--Nordstr\"{o}m \cite{cgn}; see also \cite{nordstrom}. More recently, Joyce and Karigiannis \cite{jk} introduced another construction, the input of which is a closed $G_2$-manifold with an involution.

Crowley and Nordstr\"{o}m \cite[Theorem 1.7]{cn-newinvs} show $\nu\equiv 24$ (mod 48) for twisted connected sums. This is refined by Crowley--Goette--Nordstr\"{o}m \cite[Corollary 3]{cgn}, who show $\bar{\nu}=0$ for these $G_2$-manifolds. More generally, the authors compute $\bar{\nu}$ for extra-twisted connected sums, \cite[Theorem 1]{cgn}, producing many examples with non-vanishing $\bar{\nu}$. 

Here we compute the $\nu$-invariants for many of Joyce's original $G_2$-manifolds constructed in \cite{joyce-g2}. Our investigation is by no means complete, and we largely focus on how a few observations allow one to use well-known ``soft'' properties of $\eta$-invariants in this setting. A more detailed undertaking may allow for the computation of the integer-valued $\overline{\nu}$-invariants.

The construction of a Joyce manifold $(M,g)$ involves taking the resolution of orbifold singularities in the quotient $\cO=T^7/\Gamma$ of a 7-torus $T^7$ by a finite group action $\Gamma$ preserving the flat $G_2$-structure. Many of the examples considered in \cite{joyce-g2} satisfy the following.

\begin{hypothesis}\label{hyp}
    Each connected component of the singular set in $\cO=T^7/\Gamma$ has a neighborhood isometric, for some finite subgroup $G\subset SU(2)$, to a neighborhood of the singular set in the orbifold
\begin{equation}
    T^3\times \C^{2}/G\label{eq:nbhds1}
\end{equation}
\end{hypothesis}

\noindent In general we write $T^k$ for any quotient of $\R^k$ by a discrete full rank sublattice, not necessarily $\Z^k$. When the singular set is nice enough, the $\nu$-invariant of the resulting torsion-free $G_2$-manifold $M$ may be computed from invariants of the orbifold $\cO$.

\begin{theorem}\label{thm:nuorb}
    Let $(M,g)$ be a compact $G_2$-manifold obtained from the generalized Kummer construction of Joyce using a flat orbifold $\cO=T^7/\Gamma$ satisfying Hypothesis \ref{hyp}. Then
    \begin{equation}
        \nu(M,\phi_g) \equiv 3\eta(B_{\cO})-24\eta(D_{\cO}) + 24(1+b_1(\cO)) \mod 48\label{eq:nuthm}
    \end{equation}
\end{theorem}

\noindent We will see that this result also holds under weaker conditions than those imposed by Hypothesis \ref{hyp}; see Proposition \ref{prop:nuorbext}. The right-hand side of \eqref{eq:nuthm} is straightforward to compute. The odd signature $\eta$-invariant of the orbifold $\cO=T^7/\Gamma$ may be identified with the evaluation at $s=0$ of
\begin{equation}
    \eta(B_\cO)(s) = \sum_{\lambda\neq 0}  \text{sign}(\lambda)\cdot \dim (E_\lambda^{\Gamma})\cdot |\lambda|^{-s} \label{eq:orbeta}
\end{equation}
where $\lambda$ ranges over the non-zero eigenvalues of the odd signature operator for $T^7$ with corresponding eigenspaces $E_\lambda$ and $\Gamma$-invariant subspaces $E_\lambda^\Gamma \subset E_\lambda$. In particular, we have
\begin{equation}
    \eta(B_\cO) = \frac{1}{|\Gamma|}\sum_{\gamma\in \Gamma}  \eta_\gamma(B_{T^7})\label{eq:equiveta}
\end{equation}
where $\eta_\gamma(B_{T^7})$ is the equivariant $\eta$-invariant, obtained from \eqref{eq:orbeta} by replacing $\dim(E_\lambda^\Gamma)$ with $\text{Tr}(\gamma|E_\lambda)$ and evaluating at $s=0$. Similar remarks hold for the $\eta$-invariant of the spin Dirac operator. We then compute these equivariant $\eta$-invariants using standard techniques, as in \cite{aps-ii,donnelly,mp-dirac}. We compute $\nu$ (mod $24$) for all examples in Joyce's original papers \cite{joyce-g2}, and for a majority of these we compute $\nu$ (mod $48$). Some results from our computations are:
\begin{itemize}
    \item The two simply-connected torsion-free $G_2$-manifolds with $b_2=2$ and $b_3=10$ constructed by Joyce in \cite{joyce-g2} have distinct $\nu$ invariants.
    \item For almost all examples considered, we compute $\nu\equiv 0$ (mod $24$).
    \item For all of the examples considered, we compute $\nu \equiv 0$ (mod 3).
\end{itemize}
For general remarks on the possible values for $\nu$, see e.g. \cite{cgn2}. The range of values $\nu$ takes for the examples in \cite{joyce-g2} is small; see Table \ref{table} and Figure \ref{figure:nu}. We expect that further computations, obtained in part by relaxing our assumptions on the singular set, will lead to a greater range of values. Indeed, our constraints on the singular set of $\cO$ are not necessarily optimal for Theorem \ref{thm:nuorb}, and are mainly imposed by our methods. 

The proof of Theorem \ref{thm:nuorb} compares the invariants for the orbifold $\cO$ and its resolution $G_2$-manifold through gluing formulae for $\eta$-invariants as described in \cite{bunke,kirklesch}; such formulae were used in \cite{cgn}. However, our situation does not fit the hypotheses of these formulae: our gluing region is not isometric to an interval times a hypersurface. The central observation is that the invariants under consideration are insensitive to regions of the geometry locally isometric to a product in which one factor is a flat manifold, as is the case near the singularities in $\cO$. We may then modify the geometry in these regions to satisfy the hypotheses of the gluing formulae.

A more careful analysis of the behavior of $\eta$-invariants under resolutions of orbifold singularities should lift Theorem \ref{thm:nuorb} to the integers, and compute $\overline{\nu}$. This lift is achieved in the current article for $\eta$-invariants of the odd signature operator; the case of the spin Dirac operator is more delicate. 

In the final section of the paper, we show how some of our computations can be recovered by decomposing Joyce's orbifolds along a flat 6-torus, similar to the decompositions of twisted connected sums.


Finally, we mention that a more analytical approach to computing $\overline{\nu}$ for Joyce's manifolds was studied in the PhD work of Nelvis Fornasin \cite{nelvis}.\\

\textbf{Acknowledgments.} The author would like to thank Simon Donaldson for his support and encouragement, as well as Nelvis Fornasin, Sebastian Goette and Johannes Nordstr\"{o}m for fruitful discussions. The ``twisted connected sum'' type decompositions in Section \ref{sec:tcs} were pointed out to the author by Sebastian Goette and Johannes Nordstr\"{o}m. The author was supported by the Simons Collaboration Grant on {\emph{Special Holonomy in Geometry, Analysis and Physics}}.

\section{Hypotheses on the singular set}
We first discuss Hypothesis \ref{hyp} in the general context of Joyce's construction as described in \cite[Chapter 11]{joyce-book}. Let $\Lambda$ be a lattice in $\R^7$, isomorphic as an abelian group to $\Z^7$. Then the quotient $T^7=\R^7/\Lambda$ is a 7-torus. Every point $x\in T^7$ may be written as $v+\Lambda$ for some $v\in \R^7$ and every tangent space $T_x T^7$ is naturally isomorphic to $\R^7$. The flat Euclidean $G_2$-structure on $\R^7$, described for example by the 3-form \eqref{eq:phi}, descends to a flat $G_2$-structure on $T^7$.

Let $\Gamma$ be a finite group acting on $T^7$ preserving its $G_2$-structure. Then $\cO=T^7/\Gamma$ is a flat orbifold, with an inherited flat orbifold metric $g^0$. For a subgroup $A\subset \Gamma$ let $\text{Fix}(A)$ denote the fixed points of $A$. Let $\mathcal{L}$ be the set of $F\subset T^7$ such that $F$ is a connected component of $\text{Fix}(A)$ for some subgroup $A\subset \Gamma$. Write $\mathcal{L}=\{F_i\}_{i\in I}$ where $0\in I$ is the index such that $F_0=T^7$. Then 
\[
    \text{Sing}(\cO) = \bigcup_{i\in I\setminus \{0\}} F_i\bigg/\Gamma \subset \cO
\]
forms the singular set of $\cO$. From \cite[Proposition 11.1.3]{joyce-book}, each $F_i$ with $i\neq 0$ is either a 1-torus or a 3-torus. In general, we may have two distinct 3-tori $F_i$ and $F_j$ that intersect in a 1-torus $F_k$.

The normalizer $N(F)$ of a subset $F\subset T^7$ is the subgroup of $\gamma\in \Gamma$ such that $\gamma F=F$, and the centralizer $C(F)$ is the subgroup of $\gamma\in \Gamma$ that act as the identity on $F$. For $F_i\in \mathcal{L}$ define $A_i=C(F_i)$ and $B_i=N(F_i)/C(F_i)$. Then $F_i\subset \text{Fix}(A_i)$. As $\gamma F_i$ is a component of $\text{Fix}(\gamma A_i \gamma^{-1} )$, there is an index denoted $\gamma\cdot i\in I$ with $\gamma F_i=F_{\gamma\cdot i}$. The part of $\text{Sing}(\cO)$ coming from $F_i$ is
\[
    \Gamma F_i / \Gamma = \bigcup_{\gamma\in \Gamma} F_{\gamma\cdot i} \bigg/\Gamma
\]
Now suppose $F_i\cap F_{\gamma\cdot i}=\emptyset$ for each $\gamma\in \Gamma$ such that $\gamma\cdot i \neq i$. Then $\cO$ is isomorphic near $\Gamma F_i / \Gamma$ to
\begin{equation}
    Z_i := (F_i\times W_i/A_i)/B_i \label{eq:singnbhd}
\end{equation}
Here, if $V_i$ is the invariant subspace of the action of $A_i$ lifted to $\R^7$, then $W_i$ is the orthogonal complement of $V_i$. If $B_i$ acts freely on $F_i$ then the singular locus of \eqref{eq:singnbhd} is the image of $F_i$. By \cite[Proposition 11.1.3]{joyce-book}, $W_i/A_i$ is of the form $\C^2/G$ or $\C^3/G$ for $G$ some subgroup of $SU(2)$ or $SU(3)$. Thus when $B_i$ is trivial and $\dim F_i=3$, the neighborhood \eqref{eq:singnbhd} recovers the description \eqref{eq:nbhds1}. Most of our arguments go through under the following weakening of Hypothesis \ref{hyp}:

\begin{hypothesis}\label{hyp2}
    Each connected component of the singular set in $\cO=T^7/\Gamma$ has a neighborhood isometric to a neighborhood of the singular set in the orbifold
\begin{equation}
    (T^{7-2n}\times \C^{n}/G)/B\label{eq:nbhds}
\end{equation}
where $n\in\{2,3\}$, $G$ is a finite subgroup of $SU(2)$ or $SU(3)$, the action of $B$ preserves the two factors and acts freely on the torus. Furthermore, the torus $T^{7-2n}$ admits an orientation-reversing isometry $\tau$ such that $\tau\times \text{{\emph{id}}}$ descends to define an isometry of \eqref{eq:nbhds}.
\end{hypothesis}

Hypothesis \ref{hyp2} is satisfied if: (i) $F_i\cap F_{\gamma\cdot i}=\emptyset$ whenever $\gamma\cdot i \neq i$; (ii) each $B_i$ acts freely on $F_i$; and (iii) each $F_i$ admits an orientation-reversing isometry, which by extension to the identity on $W_i/A_i$ induces an orientation-reversing isometry of $Z_i$. If $B_i$ is trivial, then (iii) automatically holds. If $B$ is $\Z/2$ and $F_i$ is a 1-torus then (iii) also holds.

\section{Flexibility of metric}
Let $(X,g)$ be a Riemannian manifold. Suppose for some open subset $U\subset X$ there is an isometry $\phi_U:(F,g_F)\times (V, g_V)\to (U,g|_U)$, where $(F,g_F)$ is flat and of dimension $\geqslant 1$, and $(V, g_V)$ is any Riemannian manifold. We say a metric $h$ on $X$ is related to $g$ by a {\emph{flat factor move}} if $g|_{X\setminus U}=h|_{X\setminus U}$ and $\phi_U^\ast(h|_U)=g_F + h_V$ for some open set $U\subset X$, isometry $\phi_U$ as above, and metric $h_V$ on $V$. We say $g$ and $h$ are {\emph{flat factor equivalent}} if they are related by a sequence of flat factor moves. Observe that it suffices in the definition to consider $(F,g_F)=(I,dt^2)$ for interals $I\subset \R$.

It was observed in \cite{aps-ii} that the odd signature $\eta$-invariant is invariant under conformal changes of the metric, and that the same is true, modulo $\Z$, for the reduced $\eta$-invariant of the spin Dirac type operator. Here the {\emph{reduced}} $\eta$-invariant of an operator $D$ is defined by
\[
    \xi(D) = \frac{1}{2}\left(\eta(D) + \dim \text{ker}(D) \right)
\]
The argument used there may be adapted to show the following:

\begin{prop}\label{prop:flatprod}
    Let $X$ be a closed oriented odd-dimensional manifold, with flat factor equivalent metrics $g$ and $h$. The odd signature $\eta$-invariants for $(X,g)$ and $(X,h)$ are equal. If $X$ is given a spin structure, the same is true for the reduced spin Dirac $\eta$-invariants taken modulo $\Z$.
\end{prop}

\begin{proof}
    It suffices to prove the claim when $h$ and $g$ are related by a flat factor move. Equip $[0,1]\times X$ with the metric $G=ds^2+ \lambda(s)h + (1-\lambda(s))g$ where $\lambda:[0,1]\to \R$ is a bump function equal to $0$ for $s\in [0,1/4]$ and equal to $1$ for $s\in [3/4,1]$. Write $B_{(M,g)}$ for the odd signature operator defined using the metric $g$. Then by the Atiyah-Patodi-Singer theorem \cite[Theorem 4.14]{aps-i} we have
    \[
        \eta(B_{(M,h)})- \eta(B_{(M,g)}) = \int_{[0,1]\times X} L(p(G))
    \]
    where $L(p(G))$ is the Hirzebruch $L$-polynomial applied to the Pontryagin forms of the Riemannian metric $G$ on $[0,1]\times X$. We use the following elementary property: the top degree term of $L(p(G))$ vanishes if the metric $G$ is a non-trivial product metric with one of its factors a flat metric. On $X\setminus U$ the metrics $g$ and $h$ are equal, and thus $G=ds^2+g|_{X\setminus U}$ has the flat factor $([0,1],ds^2)$. This implies that the integral of $L(p(G))$ over ${[0,1]\times (X\setminus U)}$ vanishes. Next,
    \[
        \int_{[0,1]\times U} L(p(G))= \int_{[0,1]\times F\times V} L(p(ds^2+g_F+\lambda(s)h_V+(1-\lambda(s))g_V))=0,
    \]
    as the metric appearing on the right, equal to $(\text{id}\times \phi_U)^\ast(G)$, has the flat factor $(F,g_F)$. The claim for the reduced $\eta$-invariant of the spin Dirac operator follows the same argument, using \cite[Theorem 4.2]{aps-i}, with the $\widehat{A}$-polynomial in place of $L$. In this latter case, the index of the Dirac operator on $[0,1]\times X$ is not a topological invariant, and so the result holds only modulo $\Z$.
\end{proof}

\noindent In conclusion, the quantities $\eta(B_M)$ and $\xi(D_M)$ (mod $\Z$) are invariants of the equivalence class of a metric generated by conformal changes and flat factor equivalences. Note that these two notions are distinct. For example, $T^2\times (S^2\setminus \text{pt})$ and $T^2\times \R^2$ with their standard metrics are flat factor equivalent but not conformally equivalent.

\section{Comparison of odd signature $\eta$-invariants}\label{sec:oddsign}
Starting from a flat $G_2$-orbifold, the construction of Joyce proceeds in two steps. First, a smooth closed Riemannian 7-manifold $(M,g^t)$ with a closed $G_2$-structure $\phi^t$ is constructed by choosing resolutions for the singularities and pasting structures together using a partition of unity. Then he shows that $g^t$ and $\phi^t$ may be deformed into a torsion-free $G_2$-structure on $M$. In this section we compare the odd signature $\eta$-invariants of the flat orbifold $(\cO,g^0)$ and $(M,g^t)$.

\begin{remark}
The parameter $t$ is any small positive real number, and roughly represents the size of the glued-in resolution pieces. However, this will only be important in Section \ref{sec:g2structure}.
\end{remark}

We first recall the relevant aspects of Joyce's construction from \cite{joyce-g2, joyce-book}. We assume Hypothesis \ref{hyp2}. We denote by $Z_i^\circ= (F_i\times D_i/A_i)/B_i$ the compact manifold with boundary obtained from $Z_i = (F_i\times W_i/A_i)/B_i$ of \eqref{eq:singnbhd} by restricting to a small closed ball $D_i\subset W_i$ centered at the origin. Let $M^\circ$ be obtained from $\cO$ by deleting neighborhoods of the singular set corresponding to the $Z_i^\circ$ and taking the closure. Then we have the decomposition
\begin{equation}
    \cO  = M^\circ\cup \bigcup_{\Gamma i \in I/\Gamma} Z_i^\circ \label{eq:odecomp}
\end{equation}
Let $i\in I$ be such that $F_i$ is a 3-torus. Choose an ALE Riemannian 4-manifold $X_i$ with holonomy $SU(2)$ which is asymptotic to $W_i/A_i$, and a free isometric action of $B_i$ on $X_i$ such that $(F_i\times X_i)/B_i$ is asymptotic to $Z_i$ in the sense of \cite[Definition 11.2.2]{joyce-book}. In particular, $X_i$ is a non-compact 4-manifold with one end, which near infinity metrically resembles the end of $W_i/A_i$. When $F_i$ is a 1-torus, we instead choose $X_i$ to be a Quasi-ALE 6-manifold with holonomy $SU(3)$ as in \cite[Chapter 9]{joyce-book}. We choose such data for each orbit $\Gamma i$. Set $M_{i} = (F_i\times X_i)/B_i$. Then $M_{i}$ is a smooth 7-manifold with one end, which may be truncated to obtain a 7-manifold with boundary, denoted $M^\circ_{i}$. The resolution manifold $M$ is then topologically
\begin{equation}
    M =  M^\circ\cup \bigcup_{\Gamma i \in I/\Gamma} M_i^\circ \label{eq:mdecomp}
\end{equation}
The decompositions \eqref{eq:odecomp} and \eqref{eq:mdecomp} are schematically depicted in Figure \ref{fig:orbifolddecomp}; there, $\cO$ is replaced by a 2-dimensional orbifold with 4 singular points, and $M$ is a corresponding resolution. Write $N_i=\partial Z_i^\circ=\partial M_i^\circ$ and $N =  \bigcup N_i=\partial M^\circ $. For small $\varepsilon$ and $t\in (0,\varepsilon]$, a closed $G_2$-structure $\phi^t$ and Riemannian metric $g^t$ on $M$ are constructed in \cite{joyce-g2, joyce-book} by patching together the torsion-free $G_2$-structures on the different pieces using a partition of unity. 

\begin{figure}[t]
\centering
\begin{tikzpicture}[scale=0.9]
\node[anchor=south west,inner sep=0] at (0,0) {\includegraphics[scale=0.675]{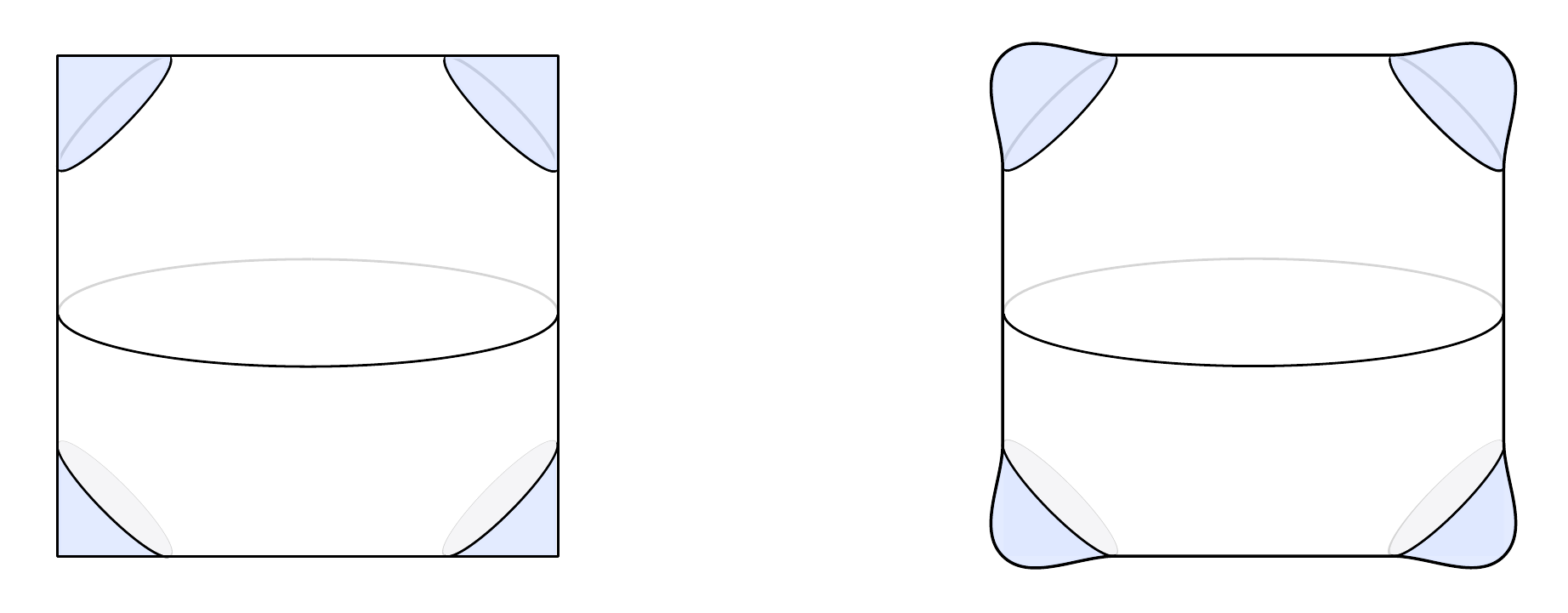}};
\node[label=above right:{$\cO$}] at (-0.75,2){};
\node[label=above right:{$M^\circ$}] at (2.25,3.5){};
\node[label=above right:{$Z_3^\circ$}] at (-0.5,-0.25){};
\node[label=above right:{$Z_2^\circ$}] at (-0.5,4.75){};
\node[label=above right:{$Z_1^\circ$}] at (5,4.75){};
\node[label=above right:{$Z_4^\circ$}] at (5,-0.25){};
\node[label=above right:{$M$}] at (13.95,2){};
\node[label=above right:{$M^\circ$}] at (10.75,3.5){};
\node[label=above right:{$M_3^\circ$}] at (7.75,-0.25){};
\node[label=above right:{$M_2^\circ$}] at (7.75,4.75){};
\node[label=above right:{$M_1^\circ$}] at (13.45,4.75){};
\node[label=above right:{$M_4^\circ$}] at (13.45,-0.25){};
\end{tikzpicture}
\caption{Schematic decompositions of $\cO$ and $M$.}\label{fig:orbifolddecomp}
\end{figure}
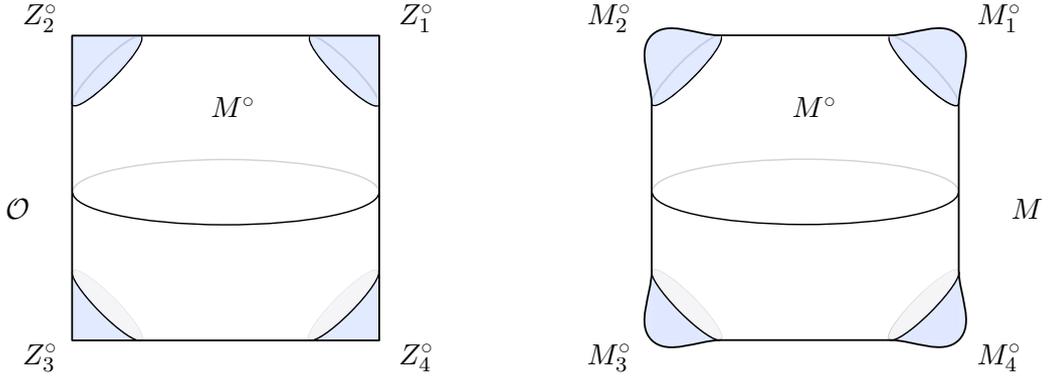

\begin{prop}\label{prop:etaequal}
    If Hypothesis \ref{hyp2} holds, then $\eta(B_{(M,g^t)})=\eta(B_{(\cO,g^0)})$.
\end{prop}

\noindent Before proceeding to the proof we make some remarks on the metrics involved, all of which are clear from Joyce's construction. We may choose the decompositions above such that $(M^o,g^t|_{M^o})$ is isometrically identified with $(M^o,g^0|_{M^\circ})$, and we may arrange that this holds for all $t\in (0,\varepsilon]$. 

The metric $g^t$ is constructed such that $(M_i^o,g^t|_{M_i^o})$ is locally the Riemannian product of metrics on $F_i$ and $X_i$. Indeed, the flat metric $g^0$ has a compatible product structure near the boundary of $M^\circ$ and thus the partition of unity respects this structure. In particular, Hypothesis \ref{hyp2} guarantees that each of $(M_i^o,g^t|_{M_i^o})$ has an orientation-reversing isometry.

To prove Proposition \ref{prop:etaequal} we invoke a gluing formula for odd signature $\eta$-invariants. As our application is similar to that of \cite{cgn}, we also refer the reader there for more details.

\begin{theorem}\cite[Theorem 8.12]{kirklesch}\label{thm:etaglue}
    Let $X$ be a closed, oriented odd-dimensional Riemannian manifold and $Y\subset X$ a hypersurface separating $X$ into $X_+$ and $X_-$. Suppose the Riemannian metric on $X$ is a product in a collar neighborhood of $Y$. Then
    \begin{equation}
        \eta(B_M) = \eta_\text{\emph{APS}}(B_{X_+},V_+) + \eta_\text{\emph{APS}}(B_{X_-},V_-) + m(V_+,V_-;H^\ast(Y)) \label{eq:etaglue}
    \end{equation}
    with real coefficients in cohomology understood, where the last term is the Maslov index of the Lagrangian subspaces $V_\pm \subset H^\ast(Y)$ defined by $V_{\pm}=\text{{\emph{im}}}(H^\ast(X_\pm)\to H^\ast(Y))$.
\end{theorem}

\noindent To apply this result, we must modify our metrics. Suppose the ball $D_i\subset W_i$ has radius $R$. In a collar neighborhood of the boundary of $M^\circ$ inside $(\cO,g^0)$, the metric is the quotient of a metric on $F_i\times W_i$ of the form $g_{1} + dr^2 + r^2g_{2}$, where $g_1$ is the metric on $F_i$ and $g_2$ is the metric on the unit sphere in $W_i$. Here $r$ is the radial coordinate in $W_i$ and $r\in [R-\epsilon,R+\epsilon]$ for some small $\epsilon>0$. This is not a product metric for the collar, but may be modified as follows.

Choose a smooth function $\rho:[R-\epsilon,R+\epsilon]\to \R$ such that $\rho(r)=r$ for $|R-r|>2\epsilon/3$ and $\rho(r)=1$ for $|R-r|<\epsilon/3$. Replace $g_{1} + dr^2 + r^2g_{2}$ with $g_{1} + dr^2 + \rho(r)^2g_{2}$; the effect of altering the metric $dr^2 + r^2g_{2}$ on a collar neighborhood of the sphere $\partial D_i\subset W_i$ to a product metric is depicted in Figure \ref{fig:collar}. This replacement is compatible with the actions of $A_i$ and $B_i$, and in this way we obtain a new metric $g^0_c$ on $\cO$ unaltered outside of small collar neighborhoods of the gluing boundaries. As the modification is radial in $W_i$, the manifold $\smash{(Z_i^o,g_c^0|_{Z_i^o})}$ retains an orientation-reversing isometry. As $F_i$ is a non-trivial flat manifold, the metric $\smash{g^0_c}$ is flat factor equivalent to $\smash{g^0}$. By Proposition \ref{prop:flatprod}, we may compute $\eta(B_\cO)$ using $g^0_c$ in place of $g^0$.

\begin{figure}[t]
\centering
\includegraphics[scale=0.75]{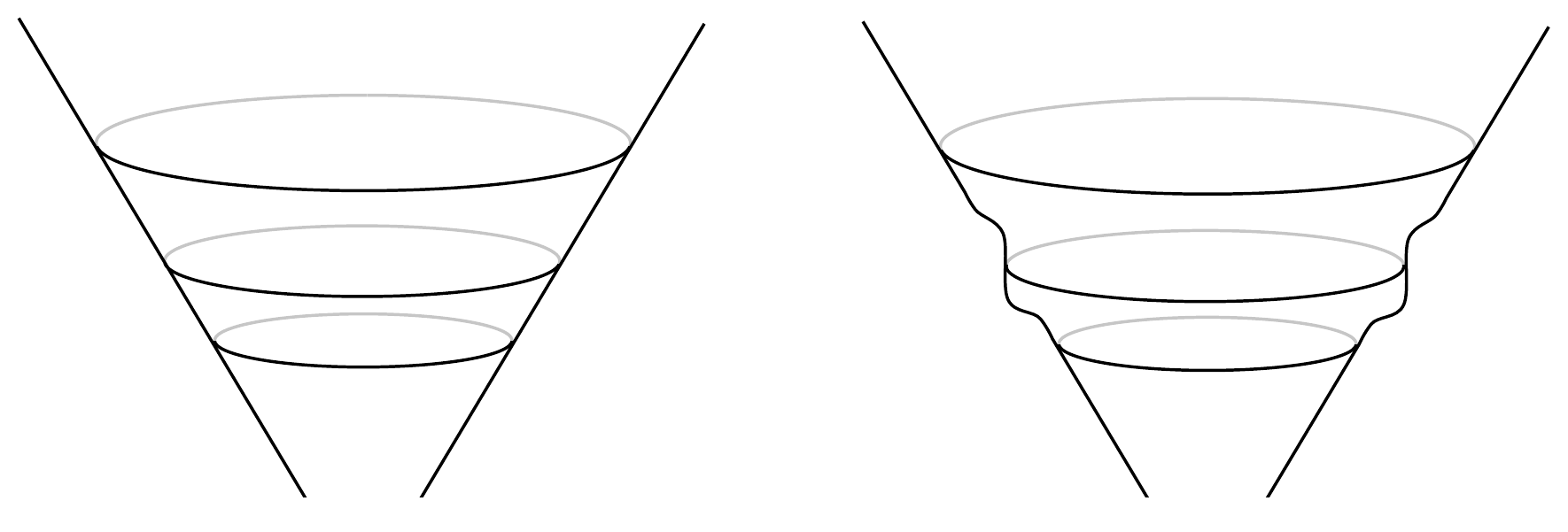}
\caption{Modifying a metric near a collar neighborhood to a product metric. }\label{fig:collar}
\end{figure}

We make a similar modification of the metric $g^t$ on $M$. In a gluing collar neighborhood as above, the metric $g^t$ is of the form $g_1+h^t$ where $h^t$ is glued together from $dr^2+r^2g_2$ and a metric on $X_i$ which may be arranged close to $dr^2+r^2g_2$. We alter $g^t$ to equal $dr^2+g_2$ on $|R-r|<\epsilon/3$, and so that it is unchanged for $|R-r|> 2\epsilon/3$ and outside of the gluing collar. This may all be done $B_i$-equivariantly, either directly or by averaging afterwards. We call the resulting metric $g^t_c$. By Proposition \ref{prop:flatprod}, we may compute $\eta(B_{(M,g^t)})$ using $g^t_c$ in place of $g^t$. We may also arrange that $(M_i^\circ,g^t_c|_{M_i^\circ})$ retains an orientation-reversing isometry.

\begin{proof}[Proof of Proposition \ref{prop:etaequal}]
    As is evident from the proof of \cite[Theorem 8.12]{kirklesch}, the statement of Theorem \ref{thm:etaglue} works equally if $X$ is a Riemannian orbifold isometric to a quotient by a finite group, as long as the hypersurface $Y$ is disjoint from the singular set. 
    
    We compare the applications of \eqref{eq:etaglue} to $(\cO,g^0_c)$ and $(M,g^t_c)$ each with $X_+=M^\circ$. The term $\eta_\text{APS}(B_{X_+},V_+)$ is the same in both applications. The term $\eta_\text{APS}(B_{X_-},V_-)$ vanishes in both cases, as $(Z^\circ_i,g^0_c|_{Z^\circ_i})$ and $(M^\circ_i,g^t_c|_{M^\circ_i})$ have orientation-reversing isometries preserving the Lagrangian subspace $V_-$. This vanishing argument is the same as in \cite[Section 4.3]{cgn}.
    
    Thus $\eta(B_{M})=\eta(B_\cO)$ holds if we show that the Maslov index $m(V_+,V_-;H^\ast(Y))$ in the two cases are the same. According to \cite[Remark 4.3]{cgn}, this Maslov index only depends on $\text{im}(H^3(X_\pm)\to H^3(Y))$. By additivity under disjoint union, it suffices to show that the images of  $H^3(Z_i^\circ)$ and $H^3(M_i^\circ)$ in $H^3(N_i)$ are equal for each $i$.
        
    There are two cases to consider. First suppose that $F_i$ is a 1-torus. We may identify $\partial D_i\subset W_i$ with $S^5$. Note $H^\ast(S^5/A_i) = H^\ast(S^5)^{A_i}=H^\ast(S^5)$. Then
    \[
        H^3(N_i)=H^3((S^1\times S^5/A_i)/B_i)= H^3(S^1\times S^5/A_i)^{B_i} = 0,
    \]
    and there is nothing to check. Next suppose $F_i$ is a 3-torus. We may identify $\partial D_i\subset W_i$ with $S^3$. Now $H^\ast(S^3/A_i)$ is isomorphic to $H^\ast(S^3)$. Thus $H^3(N_i)$ may be identified with the $B_i$-invariant subpace of $H^3(T^3)\oplus H^3(S^3/A_i)$. The map $H^3(Z^\circ_i)\to H^3(N_i)$ has image $H^3(T^3)^{B_i}$. Indeed, under these identifications it is the $B_i$-invariant image of the map
    \begin{equation}
        H^3(T^3)\oplus H^3(D^4/A_i)\to H^3(T^3)\oplus H^3(S^3/A_i).\label{eq:map1}
    \end{equation}
    The image of $H^3(M_i^\circ)$ is exactly the same; in \eqref{eq:map1}, $D^4/A_i$ is replaced by the truncation $X_i^\circ$ of $X_i$ with boundary $S^3/A_i$, and $H^3(X_i^\circ)\to H^3(S^3/A_i)$ is zero by the long exact sequence of a pair.
\end{proof}

It is clear that Proposition \ref{prop:etaequal} holds under more general conditions than stated. In particular, we have not used anything about $G_2$-structures. Similar computations may be done for resolutions of any orbifold $T^n/\Gamma$, where $n$ is aribitrary, and the singular set behaves reasonably well, as in Hypothesis \ref{hyp2}; the dimensions of the fixed point tori need not be $1$ and $3$. Note that the modification of metrics used above is valid even if $\dim F_i=0$, for in this case the alteration is conformal.

\section{Comparison of spin Dirac $\eta$-invariants}\label{sec:spin}
For the Dirac $\eta$-invariants, we follow the same strategy. We continue assuming Hypothesis \ref{hyp2} and use the setup of Section \ref{sec:oddsign}. We will apply the following mod $\Z$ gluing formula.

\begin{theorem}\cite{bunke},\cite[Theorem 5.9]{kirklesch}\label{thm:etagluedirac}
    Let $X$ be a closed, oriented odd-dimensional Riemannian manifold and $Y\subset X$ a hypersurface separating $X$ into $X_+$ and $X_-$. Suppose the metric on $X$ is a product in a neighborhood of $Y$, and the tangential operator on $Y$ is invertible. Then
    \begin{equation}
        \xi(D_X) = \xi_\text{\emph{APS}}(D_{X_+}) + \xi_\text{\emph{APS}}(D_{X_-}) \mod \Z \label{eq:diracetaglue}
    \end{equation}
\end{theorem}

We may apply Theorem \ref{thm:etagluedirac} to both $(M,g^t_c)$ and $(\cO,g^0_c)$, as the collar neighborhood in each case is a union of manifolds $[-\epsilon,\epsilon]\times N_i$, where $N_i$ is a finite quotient of either $I\times T^3\times S^3/A_i$ or $I\times S^1\times S^5/A_i$. In each case, $Y=\cup N_i$ has a metric of positive scalar curvature, and so has no harmonic spinors. Consequently, the tangential operator is invertible.

We would like to show $\xi(D_{(M,g^t)})\equiv \xi(D_{(\cO,g^0)})$ (mod $\Z$). From the gluing formula \eqref{eq:diracetaglue}, it suffices to show that $\xi_\text{APS}(D_{X_-})=0$ (mod $\Z$) in the two cases of $(M,g^t_c)$ and $(\cO,g^0_c)$. In each case, Hypothesis \ref{hyp2} guarantees that $X_-$ admits an orientation-reversing isometry. If this isometry is {\emph{spin}}, then $\eta_\text{APS}(D_{X_-})=0$. This implies $\xi_\text{APS}(D_{X_-})=h_\text{APS}(D_{X_-})/2$, where $h_\text{APS}(D_{X_-})$ denotes the dimension of the kernel of the Dirac operator $D_{X_-}$ with APS boundary conditions.

\begin{remark}
    If the orientation-reversing isometries in Hypothesis \ref{hyp2} preserve the spin structure, we say that  Hypothesis \ref{hyp2} {\emph{with spin isometries}} holds.
\end{remark}

\noindent Thus far we have established that $\xi(D_{(M,g^t)})\equiv \xi(D_{(\cO,g^0)})$ (mod $\frac{1}{2}\Z$) under the assumption of Hypothesis \ref{hyp2} with spin isometries.

Now assume Hypothesis \ref{hyp}. As $X_-$ is isometric in each case to a product $T^3\times V$ where $V$ is some Riemannian manifold, the K\"{u}nneth theorem for elliptic complexes implies $h_\text{APS}(D_{X_-})=h(D_{T^3})h_\text{APS}(D_V)$, compare \cite[p.12]{hitchin}. Harmonic spinors on $T^3$ are in correspondence with constant vectors in the standard $2$-dimensional representation of $\text{Spin}(3)$ via parallel transport. Thus $h(D_{T^3})=2$ and $h_\text{APS}(D_{X_-})\equiv 0$ (mod 2). It follows that $\xi_\text{APS}(D_{X_-})=0$ (mod $\Z$). We conclude:

\begin{prop}\label{prop:diracetaequal}
    If Hypothesis \ref{hyp} holds, then $\xi(D_{(M,g^t)})\equiv \xi(D_{(\cO,g^0)})$ {\emph{(mod $\Z$)}}. If only Hypothesis \ref{hyp2} with spin isometries holds, then this congruence holds modulo $\frac{1}{2}\Z$.
\end{prop}

\noindent However, in many cases the mod $\Z$ congruence still holds under weaker conditions than those given in Hypothesis \ref{hyp}. We next describe a common feature of many examples encountered in Section \ref{sec:comps} for which the desired congruence continues to hold.

\begin{situation}\label{sit}{\emph{
    Suppose Hypothesis \ref{hyp2} holds. Let $\{F_i/B_i\}_{i\in J}$ be the connected components of the singular set of $\cO$ which are not 3-tori. Here $J\subset I/\Gamma$. When defining the resolution manifold $M$, we replace a neighborhood $Z_i^\circ$ of each $F_i/B_i$ with $M_i^\circ$. We suppose that for each $i\in J$, there are an even number of $k\in J$ such that the resolution data $M_{k}^\circ$ is isomorphic to that of $M_i^\circ$.
    }}
\end{situation}

\begin{prop}\label{prop:diracetaequal2}
    In Situation \ref{sit}, $\xi(D_{(M,g^t)})\equiv \xi(D_{(\cO,g^0)})$ {\emph{(mod $\Z$)}}.
\end{prop}

\begin{proof}
    Let us return to the argument for the proof of Proposition \ref{prop:diracetaequal} given above, and focus on the case of $(\cO,g^0)$. We may write $X_-=\cup Z_i^\circ$ where $i$ ranges over the connected components of the singular set of $\cO$. Then $h_\text{APS}(D_{X_-})= \sum h_\text{APS}(D_{Z_i^\circ})$. The components with $F_i/B_i=T^3$ have $h_\text{APS}(D_{Z_i^\circ})\equiv 0$ (mod 2) as argued in the case of Hypothesis \ref{hyp}. Situation \ref{sit} allows us to gather the remaining $h_\text{APS}(D_{Z_i^\circ})$ into groups of equal terms of even cardinality, implying $h_\text{APS}(D_{X_-})\equiv 0$ (mod 2). The same holds for the case of $(M,g^t)$. Then $\xi(D_{(M,g^t)})\equiv \xi(D_{(\cO,g^t)})$ (mod $\Z$) follows again from the gluing formula \eqref{eq:diracetaglue}, as claimed.
\end{proof}

\section{Flexibility of $G_2$-structure}\label{sec:g2structure}
Let $(M,g)$ be a closed Riemannian 7-manifold. Suppose a spinor bundle $SM$ over $M$ is chosen, and $s\in\Gamma(SM)$ a non-vanishing spinor. Let $g^{SM}$ and $\nabla^{SM}$ denote the metric and connection on $SM$ induced by $g$ and the Levi-Civita connection of $g$. Then Crowley--Goette--Nordstr\"{o}m \cite{cgn} define $\overline{\nu}$, the integer-valued extended $\nu$-invariant of $(M,g,s)$, as follows:
\begin{equation}
    \overline{\nu}(M,g,s) = 2\int_{M} s^\ast \psi(\nabla^{SM},g^{SM}) + 3\eta(B_{M}) - 24 \eta(D_{M}) \in \Z\label{eq:nubarexp}
\end{equation}
We describe the Mathai--Quillen current $s^\ast \psi(\nabla^{SM},g^{SM})$ following \cite{cgn}. The curvature $R^{SM}$ is an element of $\Omega^2(M;\Lambda^2 SM)$, and $\nabla^{SM} s$ of $\Omega^1(M;SM)$. We have
\begin{equation}
    s^\ast\psi(\nabla^{SM},g^{SM}) = \int_0^\infty\int^B \frac{s}{2\sqrt{t}}e^{-R^{SM}+\sqrt{t}\nabla^{SM}s+t\|s\|^2}dt\label{eq:mq}
\end{equation}
Here $\int^B:\Omega^\ast(M;\Lambda^\ast SM)\to \Omega^\ast(M)$ denotes the Berezin integral, extracting a certain constant multiple of the top-degree component in $\smash{\Lambda^\ast SM}$. In particular, $s^\ast\psi(\nabla^{SM},g^{SM})$ is a differential form on $M$, not necessarily homogeneous.

As described in \cite[\S 2.3]{cn-newinvs}, unit spinors in $\Gamma(SM)$ are in correspondence with $G_2$-structures on $M$. If $\phi$ is the $G_2$-structure corresponding to $s$, then by \cite[Theorem 1.2]{cgn}, we have
\begin{equation}
    \nu(M,\phi) \equiv \overline{\nu}(M,g,s) + 24h(D_M) \mod 48\label{eq:unbartobar}
\end{equation}
for any metric $g$. Here $h(D_M)$ is the dimension of the kernel of the spin Dirac operator $D_M$. When $s$ and $g$ are determined by a $G_2$-structure $\phi$ we write $\overline{\nu}(M,\phi)=\overline{\nu}(M,g,s)$. When $s$ is $g$-parallel, or equivalently when the corresponding $G_2$-structure is compatible with $g$ and torsion-free, then it is shown \cite[Lemma 1.3]{cgn} that $s^\ast\psi(\nabla^{SM},g^{SM})=0$, leading to
\[
    \overline{\nu}(M,g) = 3\eta(B_{M}) - 24 \eta(D_{M})  \equiv \nu(M,\phi) + 24(1+b_1(M))\mod 48.
\]
Here is used the fact that $h(D_M)=1+b_1(M)$ for a closed spin Riemannian manifold with holonomy contained in $G_2$. We presently determine some other conditions under which the term $\int_M s^\ast\psi(\nabla^{SM},g^{SM})$ in \eqref{eq:nubarexp} vanishes.

We say that $(M,g,s)$ as above is {\emph{torsion-free up to flat factors}} if there is an open set $U\subset M$ such that $s$ is $\nabla^{SM}$-parallel on the complement $M\setminus U$, and $U$ is covered by open sets $U_i$ each with an isometry $\phi_i:(F_i,g_{F_i})\times (V_i,g_{V_i})\to (U_i,g|_{U_i})$, where $(F_i,g_{F_i})$ is flat and of dimension $\geqslant 1$, and the spinor $s$ is $\nabla^{SM}$-parallel in the directions $(\phi_i)_\ast(v)$ where $v\in TF_i$. In short, $(M,g)$ is a torsion-free $G_2$-manifold away from $U$ with parallel spinor $s$, and on $U$, the metric $g$ locally splits off a flat factor, and $s$ is parallel with respect to this flat factor. As in the definition of flat factor equivalence, it suffices to consider $(F,g_F)=(I,dt^2)$ for intervals $I\subset\R$.

\begin{prop}\label{prop:flatfactorphi}
    Suppose $(M,g,s)$ is torsion-free up to flat factors. Then
    \begin{equation*}
    \overline{\nu}(M,g,s) = 3\eta(B_{M}) - 24 \eta(D_{M}).
\end{equation*}
\end{prop}

\begin{proof}
    For simplicity we assume that the open covering $\{U_i\}$ consists only of $U$, and that $(F,g_F)$ is isometric to an interval $(I,dt^2)$. The computation is local on $M$ and the general case easily follows. We let $v\in\Gamma(TM|_U)$ be the vector field on $U$ induced by $\partial/\partial t$. Thus $(M,g,s)$ has $\nabla^{SM}s=0$ on $M\setminus U$, and $\nabla_v^{SM} s=0$ on $U$. By \eqref{eq:nubarexp} it suffices to show $\int_M s^\ast\psi(\nabla^{SM},g^{SM})=0$. First,
    \[
        \int_{M\setminus U} s^\ast\psi(\nabla^{SM},g^{SM})=0
    \]
    because $\smash{s^\ast\psi(\nabla^{SM},g^{SM})|_{M\setminus U}=0}$, as explained in \cite[Lemma 1.3]{cgn}. The argument is as follows. In the expression \eqref{eq:mq}, the terms $\smash{R^{SM}}$, $\smash{\nabla^{SM}s}$ and $\smash{\| s\|^2}$ have their degrees with respect to $\smash{\Lambda^\ast SM}$ given by $2$, $1$ and $0$, respectively. As $\smash{\nabla^{SM}s=0}$ on $M\setminus U$, it follows that the exponential in \eqref{eq:mq} is of even degree. This is multiplied by $s$, yielding an expression of odd degree in $\Lambda^\ast SM$. As $\text{rank}(SM)=8$, the Berezin integral vanishes, implying $s^\ast\psi(\nabla^{SM},g^{SM})|_{M\setminus U}=0$.
    
    Next, we similarly claim that the Mathai--Quillen term vanishes over $U$:
    \begin{equation}
        \int_{U} s^\ast\psi(\nabla^{SM},g^{SM})=0\label{eq:onu}
    \end{equation}
    In contrast to the argument above, it is no longer necessarily true that $s^\ast\psi(\nabla^{SM},g^{SM})|_{U}=0$. However, as the contraction of $\smash{R^{SM}}$ with $v$ is zero, and $\smash{\nabla_v^{SM}s=0}$, from \eqref{eq:mq} we easily see that the contraction of the differential form $s^\ast \psi(\nabla^{SM},g^{SM})$ with $v$ is zero. This implies that the top degree term of this differential form vanishes on $U$, from which \eqref{eq:onu} follows.
\end{proof}

The manifolds $(M,g^t)$ defined by Joyce with closed $G_2$-structures $\phi^t$, obtained by resolving an orbifold $T^7/\Gamma$ satisfying Hypothesis \ref{hyp2}, are torsion-free up to flat factors. 

\begin{corollary}\label{cor:totors}
    Let $(M,g^t)$ for $t\in (0,\epsilon]$ be a resolution of a flat $G_2$-orbifold $T^7/\Gamma$ satisfying Hypothesis \ref{hyp2}, with closed $G_2$-structure $\phi^t$ as defined by Joyce. Then
    \begin{equation}
    \overline{\nu}(M,\phi_t) = 3\eta(B_{(M,g^t)}) - 24 \eta(D_{(M,g^t)}).\label{eq:interm}
\end{equation}
If $(M,g)$ is a torsion-free $G_2$-manifold obtained from $(M,g^t,\phi^t)$ for $t\ll \epsilon$ then $\nu(M,g) \equiv \nu(M,\phi_t)$.
\end{corollary}

\begin{proof}
    Equation \eqref{eq:interm} follows from Proposition \ref{prop:flatfactorphi} and the observation that $(M,g^t,\phi^t)$ is torsion-free up to flat factors. The $G_2$-structures $\phi^t$ for $t\in (0,\epsilon]$ are all homotopic, with homotopies given by the parameter $t$. As the $\nu$-invariant is invariant under homotopies of $G_2$-structures, $\nu(M,\phi^t)$ is independent of $t\in (0,\epsilon]$. For small enough $t$, there exists a torsion-free $G_2$-structure $(M,g,\phi)$ with $\| \phi^t-\phi\|_{C^0} \leqslant K t^{1/2}$ for some constant $K$ independent of $t$, see \cite[Section 11.6]{joyce-book}. Thus for small enough $t$, $\phi$ is homotopic to $\phi^t$, and hence $\nu(M,\phi) \equiv \nu(M,\phi_t)$.
\end{proof}

\begin{proof}[Proof of Theorem \ref{thm:nuorb}]
    Combine Corollary \ref{cor:totors}, equation \eqref{eq:unbartobar}, Propositions \ref{prop:etaequal} and \ref{prop:diracetaequal}, and the following well-known observation, which has already been mentioned above: for a closed Riemannian manifold $M$ with holonomy contained in $G_2$, we have $h(D_M)=1+b_1(M)$. This is because the $G_2$-structure induces an identification of the spinor bundle with $\underline{\R}\oplus T^\ast M$, and for closed Ricci-flat manifolds, 1-forms are parallel if and only if they are harmonic. This readily adapts to the orbifold setting, so that $h(D_\cO)=1+b_1(\cO)$.
\end{proof}

\noindent In fact, from Propositions \ref{prop:diracetaequal} and \ref{prop:diracetaequal2} we have the following extension.

\begin{prop}\label{prop:nuorbext}
    Under Hypothesis \ref{hyp2} with spin isometries, the congruence \eqref{eq:nuthm} of Theorem \ref{thm:nuorb} holds modulo 24. In Situation \ref{sit}, Theorem \ref{thm:nuorb} continues to hold, i.e. \eqref{eq:nuthm} holds modulo 48.
\end{prop}

\section{Computations}\label{sec:comps}

We now use Theorem \ref{thm:nuorb} and its extensions to compute $\nu$ for many of Joyce's $G_2$-manifolds. Let $\cO=T^7/\Gamma$ be a flat $G_2$-orbifold. To compute $\nu$ using \eqref{eq:nuthm} it suffices to compute $\eta(B_\cO)$ and $\eta(D_\cO)$ (mod $2\Z$), which are averages of the equivariant invariants $\eta_\gamma(B_{T^7})$ and $\eta_\gamma(D_{T^7})$, as seen in \eqref{eq:equiveta}. A summary of our computations is given by Figure \ref{figure:nu}.

\subsection{Examples with vanishing $\eta$-invariants}

The first class of orbifolds considered by Joyce in \cite{joyce-g2} are as follows. Let $T^7=\R^7/\Z^7$. This has a flat $G_2$-structure induced by the 3-form
\begin{equation}
    \phi = dx_{127} + dx_{136} + dx_{145} +dx_{235} - dx_{246} + dx_{347} + dx_{567} \label{eq:phi}
\end{equation}
where $x_i$ are coordinates on $\R^7$ and $dx_{ijk}=dx_idx_j dx_k$. Let $\alpha$, $\beta$, $\gamma$ be the involutions
\begin{align*}
    \alpha(x_1,\ldots,x_7) & = (-x_1,-x_2,-x_3,-x_4,x_5,x_6,x_7)\\
    \beta(x_1,\ldots,x_7) & = (b_1-x_1,b_2-x_2,x_3,x_4,-x_5,-x_6,x_7)\\
    \gamma(x_1,\ldots,x_7) &= (c_1-x_1,x_2,c_3-x_3,x_4,c_5-x_5,x_6,-x_7)
\end{align*}
where $b_1$, $b_2$, $c_1$, $c_3$, $c_5 \in \{0,1/2\}$ are fixed. These involutions preserve $\phi$ and generate a group isomorphic to $(\Z/2)^3$. Examples 1--5 of \cite{joyce-g2} are obtained from resolutions of $\cO=T^7/\Gamma$ for subgroups $\Gamma \subset \langle \alpha,\beta,\gamma\rangle$ with $b_1$, $b_2$, $c_1$, $c_3$, $c_5$ fixed constants. The orientation-reversing isometry $(x_1,\ldots,x_7)\mapsto (-x_1,\ldots,-x_7)$ commutes with $\alpha$, $\beta$, $\gamma$ and reflects the eigenspaces of the Dirac operator. Thus $\eta_g(B_{T^7})=\eta_g(D_{T^7})=0$ for all $g\in \langle \alpha,\beta,\gamma\rangle$, implying $\eta(B_\cO)=\eta(D_\cO)=0$. Example 6 of \cite{joyce-g2} includes a generator sending $(x_1,\ldots,x_7)$ to $( \frac{1}{2} + x_1, x_2, \frac{1}{2} +x_3, \frac{1}{2}+x_4, \frac{1}{2}
+x_5, x_6, x_7)$; this also commutes with the above reflection, so the same holds in this case. 

For each of these examples, the singular set of $\cO$ is a disjoint union of $T^3$ and $T^3/\Z/2$ where $\Z/2$ acts by $(y_1,y_2,y_3)\mapsto (\frac{1}{2}+y_1,-y_2,-y_3)$. Thus Hypothesis \ref{hyp2} with spin isometries is satisfied, and by Proposition \ref{prop:nuorbext}, we conclude that Examples 1--6 of \cite{joyce-g2} have $\nu\equiv 0$ (mod 24).

We next consider $\nu$ (mod $48$). Each neighborhood of $T^3$ is resolved using an Eguchi--Hanson space, and each neighborhood of $T^3/\Z/2$ is resolved using an Eguchi--Hanson space with $\Z/2$-action, of which there are two choices. Let $\ell$ be the number of resolutions of one of these distinguished choices. If $\ell$ is even, we are in Situation \ref{sit}, and by Proposition \ref{prop:nuorbext}, we conclude that Examples 1--6 \cite{joyce-g2} with $\ell \equiv 0$ (mod 2) have $\nu\equiv 24(1+b_1)$ (mod 48). Note that all examples have $b_1=0$ except for Examples 1 and 2, which have $b_1=3$ and $b_1=1$, respectively; these two examples have holonomy groups strictly smaller than $G_2$.

\subsection{Donnelly's formula for $\eta(B_\cO)$}

Before proceeding to the next examples, we make some remarks. The torus $T^7=\R^7/\Lambda$ admits an orientation-reversing spin isometry, induced by negation on $\R^7$, so $\eta(B_{T^7})=\eta(D_{T^7})=0$. Next, suppose $\gamma\in \Gamma$ has non-empty, proper fixed point set $F$. Then $F$ is isometric to either a 1-torus or a 3-torus. In each of the examples we consider, it is easy to find an orientation-reversing isometry of $F$ that extends to $T^7$ which commutes with $\gamma$. Thus we may write
\[
    \eta(B_\cO) = \frac{1}{|\Gamma|}\sum_{\substack{\gamma\in\Gamma\\ \text{Fix}(\gamma)=\emptyset}} \eta_\gamma(B_{T^7})
\]
(In our examples this follows also from direct computation.) To compute each term $\eta_\gamma(B_{T^7})$ for our next set of examples, we describe a situation considered by Donnelly \cite{donnelly}. Assume $T^7$ is isometric to $T^6\times S^1$, i.e. the lattice $\Lambda\subset \R^7$ defining $T^7=\R^7/\Lambda$ splits off an orthogonal rank 1 summand, so that $\Lambda = \Lambda'\oplus \Z$ for some rank 6 lattice $\Lambda'\subset \R^6$. Suppose for $\gamma\in \Gamma$ we have
\begin{equation}
    \gamma(x,y) = (Ax+c,y+d)\label{eq:gammatorus}
\end{equation}
where $y\in S^1=\R/\Z$ and $x\in T^6$. Here $A$ acts linearly and orthogonally on $\R^6$ preserving $\Lambda'$, conjugate to the direct sum of three rotation matrices with angles $\gamma_1,\gamma_2,\gamma_3$, while $c\in T^6$ and $d\in S^1$. If $A$ has eigenvalues $\pm 1$ or $d=0\in S^1$ then $\eta_\gamma(B_{T^7})=0$. Otherwise we have 
\begin{equation*}
    \eta_\gamma(B_{T^7}) = \nu(\gamma)\cot(\pi d )\cot(\gamma_1/2)\cot(\gamma_2/2)\cot(\gamma_3/2)\label{eq:donodd1}
\end{equation*}
where $\nu(\gamma)$ is the number of fixed points of the extension of $\gamma$ to $T^6\times D^2$. This result is only a slight extension of \cite[Proposition 4.7]{donnelly}, and follows by applying the equivariant Atiyah--Patodi--Singer theorem to $\gamma$ acting on $T^6\times D^2$. Then for $\cO=T^7/\Lambda$ we may write
\begin{equation}
    \eta(B_{\cO}) = \frac{1}{|\Gamma|}\sum_{\substack{\gamma\in\Gamma\\ \text{Fix}(\gamma)=\emptyset}} \nu(\gamma)\cot(\pi d)\prod_{i=1}^3\cot(\gamma_i/2).\label{eq:donodd}
\end{equation}
Donnelly's method may be adapted to compute spin Dirac eta invariants, modulo $2\Z$, but not necessarily for the spin structure we want. We thus describe a more direct method which in addition computes the real number $\eta(D_\cO)$.

\subsection{Direct computations for $\eta(D_\cO)$}\label{sec:direct}

Spinors on $\cO$ are $\Gamma$-invariant spinors on $T^7$, which in turn are $\Lambda$-invariant spinors on $\R^7$. Following an observation which goes back to Friedrich \cite{friedrich}, the determination of eigenspinors for $\cO$ quickly becomes representation-theoretic. A direct computation of $\eta(D_\cO)$ in this fashion is nearly contained in \cite{mp-dirac}, which considers the case in which $\cO$ is smooth, but the arguments carry over to the orbifold case without difficulty. Further, as we are not considering a twisted Dirac operator, and our spin structures are naturally induced by $G_2$-structures, our situation is considerably simpler. We review the key aspects, leaving only a few details to \cite{mp-dirac}.

First, we describe eigenspaces of the Dirac operator acting on $T^7=\R^7/\Lambda$. Let $S$ denote the 8-dimensional complex spinor representation of $\text{Spin}(7)$, equipped with a Hermitian inner product $\langle \cdot , \cdot \rangle$ for which Clifford multiplication is skew-Hermitian. For $u\in\Lambda^\ast$ and $w\in S$ consider the spinor $f_{u,w}:T^7\to S$ given by $f_{u,w}(x) = e^{2\pi i \langle u, x\rangle}w$. Write $D=D_{T^7}$ for the Dirac operator. Then
\[
    D f_{u,w}(x) = \sum_{j=1}^7 e_j \cdot \frac{\partial}{\partial x_j} f_{u,w}(x) = 2\pi i u\cdot f_{u,w}(x).
\]
The Clifford relation $u^2=-|u|^2$ implies Clifford multiplication on $S$ by $u$ has eigenvalues $\pm i |u|$. Write $S^\pm_u$ for the $\mp i |u|$-eigenspaces. Then for $w\in S^\pm_u$ we have $D f_{u,w} =  \pm 2\pi |u| f_{u,w}$. By the Stone--Weierstrass Theorem, the spinors $f_{u,w}$ for $u\in\Lambda^\ast$ and $w$ ranging over bases of $S_u^\pm$ give a complete orthogonal system of $L^2(T^7; S)$. Thus the eigenspaces of $D$ are $E_{\pm\mu}$ where
\[
    E_{\pm\mu} =\{ f_{u,w}: u\in \Lambda^\ast, \mu=2\pi|u|, w\in S^{\pm}_u\}.
\]
In fact, as the model case with $u$ a multiple of $e_1$ shows, $S_u^+$ and $S_u^-$ both have dimension $4$ when $u\neq 0$, and we recover the fact that the spectrum of $D$ is symmetric.

Now we turn to $\cO=T^7/\Gamma$. Because the action of $\Gamma$ on $T^7$ respects its $G_2$-structure, we have an induced action on $S$, to be described shortly, and a resulting action on spinors $f$ defined by $(\gamma f)(x) = \gamma f(\gamma^{-1}x)$. Write $\Lambda_\mu^\ast=\{u\in \Lambda^\ast: \mu=2\pi |u|\}$ and $E_u^\pm=\{f_{u,w}: w\in S_u^\pm\}$ so that $E_{\pm \mu}=\oplus_{u\in\Lambda^\ast_\mu}E_u^\pm$. Fix $u$ and let $s^\pm_k$ be an orthonormal basis for $S_u^\pm$. Then for $\gamma\in\Gamma$ we compute
\begin{equation}
    \text{Tr}(\gamma|_{E_{u}^\pm})   \text{vol}(T^7) =\sum_{k} \langle \gamma f_{u,s_k^\pm}, f_{u,s_k^\pm}\rangle_{L^2} = \sum_k \langle\gamma s_k^\pm,  s_k^\pm\rangle \int_{T^7}  e^{2\pi i \langle u, \gamma^{-1}x\rangle - 2\pi i \langle u, x\rangle } \label{eq:tr1}
\end{equation}
Write $\gamma x = B x + b$ where $B$ is a linear map and $b\in T^7$. Noting $\gamma^{-1}x=B^{-1}x-B^{-1}b$, the integral on the right hand side of \eqref{eq:tr1} is equal to $\delta_{Bu,u}\text{vol}(T^7)e^{-2\pi i \langle u, b \rangle}$. From this we obtain
\[
    \text{Tr}(\gamma|_{E_{\pm\mu}}) = \sum_{u\in (\Lambda^\ast_\mu)^B}\text{Tr}(\gamma|_{E_{u}^\pm}) =  \sum_{u\in (\Lambda^\ast_\mu)^B}e^{-2\pi i \langle u, b \rangle}\text{Tr}(\gamma|_{S_{u}^\pm})
\]
Now we focus on $\smash{\text{Tr}(\gamma|_{S_{u}^\pm})}$. We begin by describing $S$ more explicitly. First, we recall that the 3-form $\phi$ defining the $G_2$-structure on $\R^7$ determines a cross-product $(u,v)\mapsto u\times v$ via the relation $\phi(u,v,w)=\langle u\times v, w\rangle$. We extend the cross-product complex linearly to $\C^7$. We then define $S=(\R\oplus \R^7) \otimes \C$, and Clifford multiplication for $u\in \R^7$ and $(\lambda,v)\in \C\oplus \C^7=S$ by
\[
    u\cdot (\lambda, v) = (-\langle v, u\rangle, \lambda u + u\times v),
\]
where $\langle \cdot, \cdot\rangle$ is the standard Hermitian inner product on $S=\C^8$; compare \cite[Section 10.2]{ws}. Let $e_1,\ldots,e_7$ be the standard basis for $\R^7$. With $\phi$ as in \eqref{eq:phi}, we compute that
\begin{align}
    s^\pm_0=(1, \pm i e_7)/\sqrt{2}, \quad  & s_1^\pm=(0, e_1 \pm i e_2)/\sqrt{2},\quad \label{eq:ieigen}\\ s_2^\pm=(0, e_3 \pm i e_4)/\sqrt{2},\quad & s_3^\pm=(0, e_5 \pm i e_6)/\sqrt{2} \nonumber
\end{align}
form an orthonormal basis for $S^\pm_{u}$ for $u\neq 0$ a real positive multiple of $e_7$. If $u$ is a {\emph{negative}} multiple of $e_7$ then these form bases for $S^{\mp}_u$. Note that $\gamma\in \Gamma$ acts on $S$ by $\gamma(\lambda,v)=(\lambda,B v)$. The rotation part of $\gamma$, written here as $B$, may be conjugated in $\text{Spin}(7)$ to an element $B'$ that fixes $e_7$ and rotates the $(e_k,e_{k+1})$ plane by an angle $\gamma_k$ for $k\in\{1,2,3\}$, where $\gamma_1+\gamma_2+\gamma_3\equiv 0 $ (mod $2\pi$). Also suppose that $u$, after the conjugation, is a real multiple of $e_7$, and let $\varepsilon_u = \text{sign}(u/e_7)\in\{\pm 1\}$. Then $B's_0^\pm = s_0^\pm$ and $B' s_k^\pm=e^{\mp i \gamma_k} s_k^\pm$ for $k\in\{1,2,3\}$, so that
\[
    \text{Tr}(\gamma|_{S_u^+})-\text{Tr}(\gamma|_{S_u^-}) =  \varepsilon_u \sum_{k=1}^3 e^{-i \gamma_k}-e^{i \gamma_k} = -2i\varepsilon_u\sum_{k=1}^3 \sin(\gamma_k).
\]
Recalling that $\eta_\gamma(D_{T^7})(s)=\sum_{\pm\mu\neq 0} \pm\text{Tr}(\gamma|_{E_{\pm \mu}})|\mu|^{-s}$ we obtain the formula
\begin{equation}
    \eta_\gamma(D_{T^7})(s) = -2i(2\pi)^{-s}\sum_{u\in (\Lambda^\ast\setminus 0)^B} |u|^{-s}e^{-2\pi i\langle u,b\rangle}\varepsilon_u \sum_{k=1}^3 \sin(\gamma_k).\label{eq:eqetadir}
\end{equation}
In our next set of examples, every $u\in (\Lambda^\ast\setminus 0)^B$ for $\gamma$ with $\text{Fix}(\gamma)=\emptyset$ is an integral multiple of $e_7$, and this formula simplifies considerably. Note that $\varepsilon_{-u}=-\varepsilon_u$. We also mention that the same approach described here may be used to compute the odd signature $\eta$-invariant.

\subsection{Dihedral examples}\label{sec:dihedral}

We now consider Examples 7--14 of \cite{joyce-g2}. The general setup is as follows. Let $z_1,z_2,z_3$ be coordinates for $\C^3$ and $x$ for $\R$. Let $\Lambda'\subset \C^3$ be a rank 6 lattice. Then the 7-torus $T^7=\C^3\times \R/\Lambda'\times \Z$ has a flat $G_2$-structure induced by the 3-form
\[
    \phi = \omega\wedge dx + \text{Im}(\Omega)
\]
where $\omega = \frac{i}{2}\sum_{k=1}^3 dz_k d\overline{z}_k$ and $\Omega = dz_1dz_2dz_3$. Let $u$ and $v$ be complex roots of unity, and $a$ the smallest positive integer such that $u^a=v^a=1$. Let $\alpha$ and $\beta$ be the isometries of $\C^3\times \R$ defined by
\begin{align*}
    \alpha(z_1,z_2,z_3,x) & = (uz_1,vz_2,\overline{uv}z_3,x+\tfrac{1}{a})\\
    \beta(z_1,z_2,z_3,x) & = (-\overline{z}_1,-\overline{z}_2,-\overline{z}_3,-x)
\end{align*}
If $\alpha$ and $\beta$ preserve $\Lambda'$, they descend to isometries on $T^7$ that preserve $\phi$. They satisfy $\alpha^a=\beta^2=1$ and $\alpha\beta=\beta\alpha^{-1}$, and thus generate a dihedral group of order $2a$:
\begin{equation}
    \Gamma := \langle \alpha,\beta \rangle = \left\{1,\alpha,\alpha^2,\ldots,\alpha^{a-1},\beta,\beta\alpha,\beta\alpha^2,\ldots,\beta\alpha^{a-1} \right\}\label{eq:dihedralgroup}
\end{equation}
The elements $\beta\alpha^j$ all have nonempty fixed point sets, while $\alpha^j$ for $j\not\equiv 0$ (mod $a$) has no fixed points. If we write $\alpha$ as in \eqref{eq:gammatorus}, then $c=0$ and $d=1/a$, and the rotation angles $\theta_1,\theta_2,\theta_3$ of $A$ are determined by $u=e^{i\theta_1}$, $v=e^{i\theta_2}$ and $\theta_1+\theta_2+\theta_3\equiv 0$ (mod $2\pi$). Upon computing $\smash{\nu(\alpha^j)=\det(1-A^j)=64\prod_{k=1}^3 \sin^2(j\theta_k/2)}$, from \eqref{eq:donodd} we obtain
\begin{align}
        \eta(B_\cO)
    & = \frac{4}{a}\sum_{j=1}^{a-1}\cot(j\pi/a)\prod_{k=1}^3 \sin(j\theta_k) \nonumber \\
     & = -\frac{1}{a}\sum_{j=1}^{a-1} \cot(j\pi/a)\sum_{k=1}^3 \sin(2j\theta_k) = 2 \sum_{k=1}^3 ((\theta_k/\pi)).\label{eq:etasignalmost}
\end{align}
The second equality follows from the elementary identity $-4\prod_{k=1}^3\sin(\theta_k)=\sum_{k=1}^3 \sin(2\theta_k)$, which holds whenever $\theta_1+\theta_2+\theta_3\equiv 0 $ (mod $2\pi$), and the third equality follows from the classical identity of Eisenstein \cite{eisenstein}, which says that $-\frac{1}{2a}\sum_{j=1}^{a-1}\cot(\pi j /a)\sin(2\pi jx/a)=((x/a))$ where $((t))=t-[t]-1/2$ if $t\not\in\Z$ and $((t))=0$ otherwise. Let us normalize the angles $\theta_j$ such that $\theta_j\in [0,2\pi)$ for each $j$. Then equation \eqref{eq:etasignalmost} yields
\begin{equation}
        \eta(B_\cO)
     = \begin{cases} +1, & \theta_i < \pi \text{ for } i=1,2,3\\ -1, & \theta_i > \pi \text{ for some } i\\
    \phantom{+} 0, & \theta_i\in \{0,\pi\} \text{ for some }i\end{cases} \label{eq:donoddex}
\end{equation}
We now turn to the Dirac eta invariants. In \eqref{eq:eqetadir}, when $\gamma=\alpha^j$ every $u\in(\Lambda^\ast)^B$ is an integer multiple of $e_7$. We sum over $u=\pm n e_7$ with $n\in \Z_{>0}$ and $\varepsilon_{u}=\pm 1$ to obtain 
\begin{equation}
    \eta_{\alpha^j}(D_{T^7})(s) = -4 (2\pi)^{-s} \sum_{n=1}^\infty \frac{1}{n^s} \sin(2\pi jn/a) \sum_{k=1}^3 \sin(j\theta_k).\label{eq:etazeta1}
\end{equation}
The function $\sum_{n=1}^\infty \sin(2\pi jn/a)n^{-s}$ is the imaginary part of the polylogarithm $L_s(z)=\sum_{n=1}^\infty z^n/n^s$ evaluated at $z=e^{2\pi i j/a}$. Using the identity $L_0(z)=z/(1-z)$, we evaluate \eqref{eq:etazeta1} at $s=0$:
\begin{equation*}
    \eta_{\alpha^j}(D_{T^7}) = -2\cot(\pi j /a)\sum_{k=1}^3 \sin(j\theta_k).
\end{equation*}
Finally, taking the average over the equivariant $\eta$-invariants we obtain
\begin{equation}
    \eta(D_\cO) = -\frac{1}{a}\sum_{j=1}^{a-1} \cot( \pi j/a) \sum_{k=1}^3\sin(j\theta_k)=\begin{cases} -1, & \theta_i \not\equiv 0 \;(\text{mod } 2\pi) \text{ for } i=1,2,3\\ \phantom{-}0, & \text{ otherwise}\end{cases} \label{eq:donspinex}
\end{equation}
where the second equality follows again from Eisenstein's identity. We may now compute the $\nu$-invariants of the torsion-free $G_2$-manifolds $(M,\phi)$ as constructed in Examples 7--14 of \cite{joyce-g2}. Each of these satisfies Hypothesis \ref{hyp}, and is either constructed from a dihedral orbifold as just described, or is a finite quotient thereof.\\

\begin{figure}[t]
\centering
\begin{tabular}{l|r|r|r|r|r|r|r}
\hline
 Ex. No. & $|\Gamma|$ & $\pi_1$ & $b_2$  & $b_3$ & $\eta(B_{M})$ & $\overline{\eta}(D_M)$ & $\nu$ (mod $48$) \\
 \hline
 $7$ & $6$ & $0$ & $5$ & $13$ & $1$ & $-1$ & $3$\\
 $8$ & $12$ & $0$ & $3$ & $11$ & $-1$ & $-1$ & $45$\\
 $9$ & $8$ & $0$ & $11$ & $36$ & $0$ & $-1$ & $0$\\
 $10$ & $16$ & $\Z/2$ & $6$ & $21$ & $0$ & $-1$ & $0$\\
  $11$ & $12$ & $0$ & $4$ & $17$ & $0$ & $-1$ & $0$\\
  $12$ & $24$ & $\Z/2$ & $2$ & $11$ & $0$ & $-1$ & $0$\\
  $13$ & $14$ & $0$ & $2$ & $10$ & $-1$ & $-1$ & $45$\\
  $14$ & $18$ & $0$ & $2$ & $10$ & $1/3$ & $-1/3$ & $33$\\
 \hline
\end{tabular}
\captionof{table}[table]{$\nu$ invariants for Examples 7--14 of \cite{joyce-g2}. Here $\overline{\eta}:=\eta$ (mod $2\Z$).}\label{table}
\end{figure}

\noindent {\emph{Example 7 of \cite{joyce-g2}}}. This example has $a=3$, $u=v=e^{2\pi i/3}$ and $\Lambda' = \Z^3\oplus e^{2\pi i/3}\Z^3\subset \C^3$. Thus $\theta_1=\theta_2=\theta_3=2\pi/3$. Then \eqref{eq:donoddex} and \eqref{eq:donspinex} yield $\eta(B_\cO)=1$ and $\eta(D_\cO)=-1$. For the resulting resolution torsion-free $G_2$-manifold $(M,\phi)$ we have $\nu(M,\phi)\equiv 3$ (mod $48$).\\

\noindent {\emph{Example 8 of \cite{joyce-g2}}}. Here $a=6$, $u=v=e^{\pi i/3}$ and $\Lambda'$ is as in Example 7. Thus $\theta_1=\theta_2=\pi/3$ and $\theta_3=4\pi/3$. Then $\eta(B_\cO)=\eta(D_\cO)=-1$, and $\nu(M,\phi)\equiv 45$ (mod $48$).\\

\noindent {\emph{Example 9 of \cite{joyce-g2}}}. Here $a=4$, $u=v=i$ and $\Lambda'=\Z^3\oplus i\Z^3\subset \C^3$. Thus $\theta_1=\theta_2=\pi/2$ and $\theta_3=\pi$. Then $\eta(B_\cO)=0$, $\eta(D_\cO)=-1$, and $\nu(M,\phi)\equiv 0$ (mod $48$).\\

\noindent {\emph{Example 10 of \cite{joyce-g2}}}. This is obtained from Example 9 by incorporating the involution $\gamma$ defined by $(z_1,z_2,z_3,x)\mapsto (z_1,z_2,z_3+\frac{1+i}{2},x)$. Thus the computations are slightly different: when averaging over the equivariant $\eta$-invariants, in addition to the terms associated to $\alpha^j$ are those associated to $\gamma\alpha^j$. It is easily verified that $\eta_{\gamma\alpha^j}(B_{T^7})=\eta_{\alpha^j}(B_{T^7})$, and similarly for the Dirac operator. As the size of our group has doubled, we obtain the same results as in Example 9.\\

\noindent {\emph{Example 11 of \cite{joyce-g2}}}. This example has $a=6$, $u=e^{\pi i/3}$, $v=e^{2\pi i/3}$ with corresponding lattice $\Lambda' = (\Z\oplus e^{2\pi i/3}\Z)\oplus (\Z\oplus e^{2\pi i/3}\Z)\oplus (\Z\oplus i\Z)\subset \C^3$. Thus $\theta_1=\pi/3$, $\theta_2=2\pi/3$, $\theta_3=\pi$. We have $\eta(B_\cO)=0$, $\eta(D_\cO)=-1$ and $\nu(M,\phi)\equiv 0$ (mod $48$).\\

\noindent {\emph{Example 12 of \cite{joyce-g2}}}. This is obtained from Example 11 by incorporating the involution defined in Example 10, and yields the same results as in Example 11, similar to Example 10.\\

\noindent {\emph{Example 13 of \cite{joyce-g2}}}. This has $a=7$, $u=e^{2\pi i/7}$, $v=u^2$ and
\[
    \Lambda' =\langle (u^j,u^{2j},u^{4j})\in \C^3: j=1,2,3,4,5,6\rangle.
\]
Thus $\theta_1=2\pi/7$, $\theta_2=4\pi/7$, $\theta_3=8\pi/7$; thus $\eta(B_\cO)=\eta(D_\cO)= -1$, and $\nu(M,\phi)\equiv 45$ (mod $48$).\\

\noindent {\emph{Example 14 of \cite{joyce-g2}}}. Consider the dihedral group $\langle \alpha,\beta\rangle $ defined in Example 7, and adjoin to it $\gamma(z_1,z_2,z_3,x)=(e^{2\pi i/3}z_1,e^{4\pi i/3}z_2,z_3+i/\sqrt{3},x)$. Then $\alpha$ and $\gamma$ commute, while $\gamma\beta = \beta\gamma^{-1}$. Thus $\Gamma=\langle \alpha,\beta,\gamma\rangle$ is a group of order 18. The elements with fixed points are $\beta\alpha^j\gamma^k$. In addition, the elements $g=\alpha^j\gamma^k$ with $k\in\{1,2\}$ have $\eta_g(B_\cO)=0$ and $\eta_g(D_{T^7})\equiv 0$ (mod $2\Z$), as each such $g$ has more than one $+1$ eigenvalue in its associated rotation matrix $A$ from \eqref{eq:gammatorus}. Thus the computations of the $\eta$-invariants differ from Example 7 only in that $|\Gamma|=18$ instead of $|\Gamma|=6$. We have $\eta(B_\cO)=1/3$ and $\eta(D_\cO)=-1/3$, implying $\nu\equiv 33$ (mod $48$). 

\begin{figure}[t]
\centering
\begin{tikzpicture}
	\begin{axis}[%
	xlabel={$b_3$}, ylabel={$b_2$}, grid, x post scale=2,
	every axis x label/.style={
	at={(ticklabel* cs:1.02)}, below=2.5mm, anchor=west},
	every axis y label/.style={
	at={(ticklabel* cs:1.07)}, left=5.5mm, anchor=west},
	scatter/classes={%
		a={black},
		b={draw=black, fill=white},
		c={fill=green, mark size=2.5}, 
		f={mark=square*,mark size=3, fill=yellow}, 
		d={mark=triangle*, mark size=3.75, fill=cyan}, 
		e={mark=pentagon*, mark size=3.5, fill=red}}] 
	\addplot[scatter,only marks,
		scatter src=explicit symbolic]%
	table[meta=label] {
x     y      label
43	12	a		
47	8	a
46	9	b
45	10	a
44	11	b
43	12	a
42	13	b
41	14	a
40	15	b
39	16	a
35	4	a
34	5	b
33	6	a
32	7	b
31	8	a
30	9	b
29	10	a
28	11	b
27	12	a
29	2	a
28	3	b
27	4	a
26	5	b
25	6	a
24	7	b
23	8	a
22	9	b
21	10	a
22	1	a
21	2	b
20	3	a
19	4	b
18	5	a
17	6	b
16	7	a
19	0	a
18	1	b
17	2	a
16	3	b
15	4	a
14	5	b
13	6	a
15	0	a
14	1	b
13	2	a
12	3	b
11	4	a
13	5	d
11	3	f
36	11	c
21	6	c
17	4	c
11	2	c
10	2	e
17	6	b
18	5	b
6	3	b
7	2	b
7	8	b
4	3	b
	};
\end{axis}
\end{tikzpicture}
\caption{This is a variant of \cite[Table 2]{joyce-g2} listing the betti numbers $b_2$, $b_3$ of all the closed manifolds constructed in \cite{joyce-g2} with holonomy $G_2$. Our computations are: \,\protect\nodea\;: $\nu \equiv 0$ (mod $48$); \,\protect\nodeb\;: $\nu \equiv 0$ (mod $24$); \,\protect\nodec\;: $\nu \equiv 24$ (mod 48); \protect\noded\;: $\nu\equiv 3$ (mod $48$); \,\protect\nodee\;: $\nu\equiv 45$ (mod $48$); \,\protect\nodef\;: there are two $G_2$-manifolds here, with $\nu\equiv 45$ and $\nu \equiv 33$ (mod $48$). Some nodes have more than one manifold; apart from \,\protect\nodef\,, the values of $\nu$ hold for all $G_2$-manifolds from \cite{joyce-g2} at the given node.}\label{figure:nu}
\end{figure}

\subsection{A few more examples}

The remaining examples in \cite{joyce-g2} satisfy Hypothesis \ref{hyp2} with spin isometries, and thus by Proposition \ref{prop:nuorbext} we may compute $\nu$ (mod 24). In fact, the orientation-reversing isometry $(x_1,\ldots,x_7)\mapsto (-x_1,\ldots,-x_7)$ is well-defined on $\cO=T^7/\Lambda$ for each of Exs. 15--18 of \cite{joyce-g2}, and commutes with $\Gamma$ in each case. Thus $\eta_\gamma(B_\cO)=0$ and $\eta_\gamma(D_\cO)\equiv 0$ (mod $\Z$) for each $\gamma\in \Gamma$, the latter holding because the isometry reverses the Dirac spectrum. Thus $\nu \equiv 0$ (mod $24$) for Exs. 15--18 of \cite{joyce-g2}. A summary of all our computations is represented in Figure \ref{figure:nu}.

\section{Twisted connected sum type decompositions}\label{sec:tcs}
Here we explain another route to some of the above computations. The starting point is as follows: suppose we can realize a given Joyce orbifold $\cO=T^7/\Gamma$ as a union
\begin{equation}
    \cO = \cO_{+}\cup \cO_- \label{eq:tcsdecomp}
\end{equation}
where $\cO_{\pm}$ are orbifolds with (smooth) boundaries whose collar neighborhoods are isometric to $T^6\times [-\epsilon,\epsilon]$, and $\cO$ is obtained by gluing along the $T^6$ boundaries. See Figure \ref{fig:orbifolddecomptcs}. This is similar to the twisted connected sum picture, with cross-section $T^6$ replacing $K3\times T^2$. Assume $b_1(\cO)=0$, and that $\cO$ satisfies Hypothesis \ref{hyp} or one of its variations. Then we have shown that the resolution $G_2$-holonomy manifold $(M,g)$ as constructed by Joyce has $\nu$-invariant given by
\[
        \nu(M,\phi_g) \equiv 3\eta(B_{\cO})-24\eta(D_{\cO}) + 24 \mod 48
\]
Instead of computing the terms directly for $\cO$ as before, we may apply gluing formulas for $\eta(B_\cO)$ and $\eta(D_\cO)$ using the decomposition \eqref{eq:tcsdecomp}. We first explain this for the signature operator term.

\begin{figure}[t]
\centering
\begin{tikzpicture}[scale=1]
\node[anchor=south west,inner sep=0] at (0,0) {\includegraphics[scale=0.675]{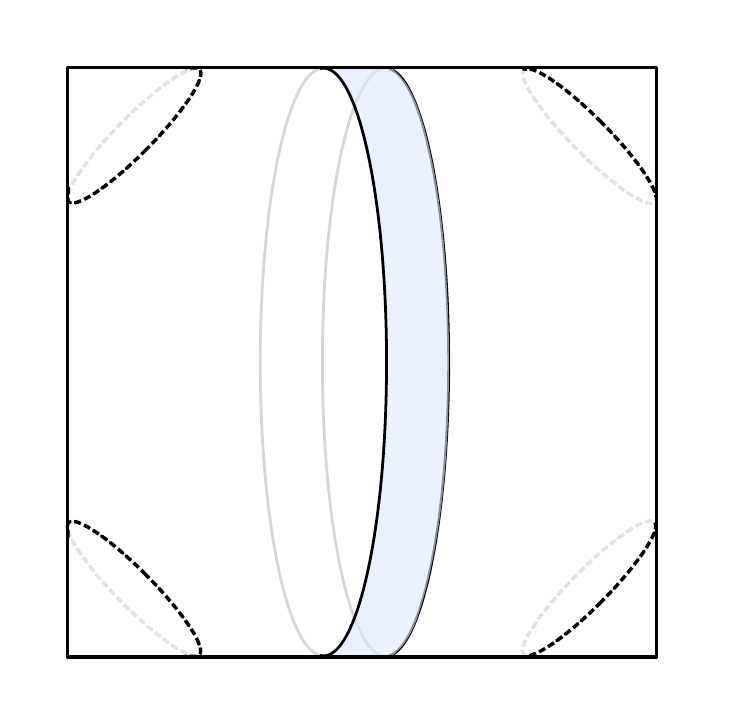}};
\node[label=above right:{$\cO_+$}] at (.75,2){};
\node[label=above right:{$T^6\times [-\epsilon,\epsilon]$}] at (1.4,-.8){};
\draw (2.6,0) -- (2.6,0.75); 
\node[label=above right:{$\cO_-$}] at (3.3,2){};
\end{tikzpicture}
\caption{A twisted connected sum type decomposition of $\cO$.}\label{fig:orbifolddecomptcs}
\end{figure}

Write $L_\pm\subset H^3(T^6)$ for the two Lagrangians given by the images of $H^3(\cO_\pm)\to H^3(T^6)$. Assume $\cO_\pm$ each admit an orientation-reversing isometry. Then Theorem \ref{thm:etaglue} yields
\[
    \eta(B_{\cO}) = m(L_+,L_-;H^3(T^3))
\]
as the relative $\eta$-invariants vanish by spectral symmetry. The Maslov index is computed as follows, see \cite[Section 4.2]{cgn}. Let $A_\pm$ be the isometries of $H^3(T^6)$ which anticommute with the Hodge star whose 1-eigenspaces are $L_\pm$. Let $E_-\subset H^3(T^6;\C)$ be the $(-i)$-eigenspace of the Hodge star. Write the eigenvalues of $-A_+A_-|_{E_{-}}$ as $e^{i\phi_1},\ldots,e^{i\phi_{10}}$ where $\phi_j\in (-\pi,\pi]$. Then:
\begin{equation}
    m(L_+,L_-;H^3(T^6)) = - \sum_{\phi_j\neq \pi} \frac{\phi_j}{\pi} \label{eq:maslovangles}
\end{equation}

We apply this to some of the dihedral examples from Section \ref{sec:dihedral}, so that $\cO=T^7/\Gamma$ where $\Gamma$ is the dihedral group \eqref{eq:dihedralgroup}. Here we have a decomposition as in \eqref{eq:tcsdecomp} with
\[
    \cO_+ = \left(T^6\times [-\tfrac{1}{4a},\tfrac{1}{4a}] \right)/ \beta , \qquad  \cO_- = \left(T^6\times [\tfrac{1}{4a},\tfrac{3}{4a}] \right)/ \alpha\beta 
\]
This was pointed out to the author by Sebastian Goette and Johannes Nordstr\"{o}m. Write the coordinates of $T^6$ as $(z_1,z_2,z_3)$ as in Section \ref{sec:dihedral}. We may identify $H^3(T^6)=H^3(T^6;\R)$ with the Hodge groups $H^{3,0}(T^6)\oplus H^{2,1}(T^6)\subset H^3(T^6;\C)$. Concretely,
\[
    H^3(T^6)= \bigoplus_{jk\ell\in S}\R\cdot \text{Re}(dz_{jk\ell})\oplus \R\cdot\text{Im}(dz_{jk\ell})
\]
where $S=\{123, \overline{1}23,1\overline{2}3,12\overline{3},1\overline{1}2,1\overline{1}3,12\overline{2},2\overline{2}3,13\overline{3},23\overline{3}\}$. The Lagrangian $L_+$ is the subspace of $H^3(T^6)$ invariant under $\beta$. We compute
\[
    L_+ = \bigoplus_{jk\ell\in S}\R\cdot \text{Im}(dz_{jk\ell})
\]
Similarly, $L_-$ is the subspace invariant under $\alpha\beta$. It is spanned by $\text{Im}(dz_{123})$, the 3 elements
\[
    \text{Im}(e^{-i\theta_j}dz_{\overline{j}k\ell})= \cos(\theta_j)\text{Im}(dz_{\overline{j}k\ell})-\sin(\theta_j)\text{Re}(dz_{\overline{j}k\ell})
\]
where $j,k,\ell\in\{1,2,3\}$ are distinct, and the 6 elements
\[ 
    \text{Im}(e^{i\theta_k/2}dz_{j\overline{j}k})=
 \cos(\theta_k/2)\text{Im}(dz_{j\overline{j}k})+\sin(\theta_k/2)\text{Re}(dz_{j\overline{j}k})
\]
where $j,k\in \{1,2,3\}$ are distinct. Note each of $L_\pm$ is 10-dimensional, coherent with the fact that they are Lagrangian subspaces inside the 20-dimensional space $H^3(T^6)$.

Turning to the Maslov index, a computation from the above description shows that the 10 eigenvalues of $-A_+A_-|_{E_{-}}$ are $-1$, $-e^{-2i\theta_j}$, where $j\in \{1,2,3\}$, and 3 conjugate pairs. Thus \eqref{eq:maslovangles} is
\[
    -\sum_{j=1}^3 \frac{\phi_j}{\pi} 
\]
where $\phi_j$ is the value in $(-\pi,\pi]$ equal to $\pi-2\theta_j$ modulo $2\pi$, as long as all $\theta_j\not\in \{0,\pi\}$; otherwise \eqref{eq:maslovangles} is zero. This is the same as expression \eqref{eq:donoddex}, our previous computation of $\eta(B_\cO)$. In particular, we recover $\eta(B_\cO)$ for Examples 7, 8, 9, 11, 13 of \cite{joyce-g2}, previously computed in Table \ref{table}.

The same procedure can be carried out for the spin Dirac invariant $\eta(D_\cO)$. Assume, as is the case in the above examples, that $\cO_\pm$ admit orientation-reversing spin isometries. Then the relative $\eta$-invariants vanish, and the gluing formula \cite[Theorem 1.8]{bunke} yields 
\[
    \eta(D_\cO) \equiv m(S_+,S_-;S_{T^6}) \;\; \text{mod }\Z
\]
Here $S_{T^6}$ is the space of harmonic spinors on the cross-section $T^6$, which may be identified with parallel spinors on $T^6\times (-\epsilon,\epsilon)$. The subspaces $S_\pm$ consist of those spinors that extend to (orbifold) harmonic spinors over $\cO_\pm$. The space $S_{T^6}$ may be identified with the $\text{Spin}(6)$ representation $\R\oplus \R^7$ obtained by restriction from the $\text{Spin}(7)$ representation in Section \ref{sec:direct}. As before, $e_i$ are the standard basis vectors for $\R^7$. Clifford multiplication by $e_7$ plays the role of the Hodge star here. The Lagrangian $S_+$ is the $\beta$-invariant subspace, which is spanned by
\[
    (1,0), \;\; (0,e_2), \;\; (0,e_4), \;\; (0,e_6).
\]
Next, $S_-$ is the $\alpha\beta$-invariant subspace, spanned by $(1,0)\in S_{T^6}=\R\oplus\R^7$ and the 3 elements
\begin{gather*}
(0,-\sin(\theta_1/2)e_1 + \cos(\theta_1/2)e_2),\\ (0,-\sin(\theta_2/2)e_3 + \cos(\theta_2/2)e_4),\\ (0,-\sin(\theta_3/2)e_5 + \cos(\theta_3/2)e_6).
\end{gather*}
The $(-i)$-eigenspace $E_-\subset S_{T^6}\otimes \C$ in this case is spanned by $s_j^+$ for $j\in\{1,2,3,4\}$, from \eqref{eq:ieigen}. The matrix $-A_+A_-|_{E_-}$ has the $4$ eigenvalues $-1$ and $e^{i(\pi-\theta_j)}$ for $j\in\{1,2,3\}$. Then
\[
    m(S_+,S_-;S_{T^6})= -\sum_{j=1}^3 \frac{\pi-\theta_j}{\pi} = -1
\]
as long as the $\theta_j$ are not all in $2\pi \Z$, and otherwise it vanishes. This recovers \eqref{eq:donspinex} modulo $\Z$, although we see that the answers agree as integers as well, which suggests that the mod $\Z$ restriction on the gluing formula may not be neccessary in this setting.

\bibliography{references}

\begin{bibdiv}
\begin{biblist}

\bib{aps-i}{article}{
      author={Atiyah, M.~F.},
      author={Patodi, V.~K.},
      author={Singer, I.~M.},
       title={Spectral asymmetry and {R}iemannian geometry. {I}},
        date={1975},
        ISSN={0305-0041},
     journal={Math. Proc. Cambridge Philos. Soc.},
      volume={77},
       pages={43\ndash 69},
         url={https://doi.org/10.1017/S0305004100049410},
      review={\MR{0397797}},
}

\bib{aps-ii}{article}{
      author={Atiyah, M.~F.},
      author={Patodi, V.~K.},
      author={Singer, I.~M.},
       title={Spectral asymmetry and {R}iemannian geometry. {II}},
        date={1975},
        ISSN={0305-0041},
     journal={Math. Proc. Cambridge Philos. Soc.},
      volume={78},
      number={3},
       pages={405\ndash 432},
         url={https://doi.org/10.1017/S0305004100051872},
      review={\MR{0397798}},
}

\bib{bunke}{article}{
      author={Bunke, Ulrich},
       title={On the gluing problem for the {$\eta$}-invariant},
        date={1995},
        ISSN={0022-040X},
     journal={J. Differential Geom.},
      volume={41},
      number={2},
       pages={397\ndash 448},
         url={http://projecteuclid.org/euclid.jdg/1214456222},
      review={\MR{1331973}},
}

\bib{cgn}{article}{
      author={Crowley, Diarmuid},
      author={Goette, Sebastion},
      author={Nordstr\"om, Johannes},
       title={An analytic invariant of {$G_2$}-manifolds},
        date={2018},
        note={https://arxiv.org/abs/1505.02734},
}

\bib{cgn2}{article}{
      author={Crowley, Diarmuid},
      author={Goette, Sebastion},
      author={Nordstr\"om, Johannes},
       title={Dinstinguishing {$G_2$}-manifolds},
        date={2018},
        note={https://arxiv.org/abs/1808.05585},
}

\bib{chnp-weak}{article}{
      author={Corti, Alessio},
      author={Haskins, Mark},
      author={Nordstr\"om, Johannes},
      author={Pacini, Tommaso},
       title={Asymptotically cylindrical {C}alabi-{Y}au 3-folds from weak
  {F}ano 3-folds},
        date={2013},
        ISSN={1465-3060},
     journal={Geom. Topol.},
      volume={17},
      number={4},
       pages={1955\ndash 2059},
         url={https://doi.org/10.2140/gt.2013.17.1955},
      review={\MR{3109862}},
}

\bib{chnp}{article}{
      author={Corti, Alessio},
      author={Haskins, Mark},
      author={Nordstr\"om, Johannes},
      author={Pacini, Tommaso},
       title={{${G}_2$}-manifolds and associative submanifolds via semi-{F}ano
  3-folds},
        date={2015},
        ISSN={0012-7094},
     journal={Duke Math. J.},
      volume={164},
      number={10},
       pages={1971\ndash 2092},
         url={https://doi.org/10.1215/00127094-3120743},
      review={\MR{3369307}},
}

\bib{cn-newinvs}{article}{
      author={Crowley, Diarmuid},
      author={Nordstr\"om, Johannes},
       title={New invariants of {$G_2$}-structures},
        date={2015},
        ISSN={1465-3060},
     journal={Geom. Topol.},
      volume={19},
      number={5},
       pages={2949\ndash 2992},
         url={https://doi.org/10.2140/gt.2015.19.2949},
      review={\MR{3416118}},
}

\bib{donnelly}{article}{
      author={Donnelly, Harold},
       title={Eta invariants for {$G$}-spaces},
        date={1978},
        ISSN={0022-2518},
     journal={Indiana Univ. Math. J.},
      volume={27},
      number={6},
       pages={889\ndash 918},
         url={https://doi.org/10.1512/iumj.1978.27.27060},
      review={\MR{511246}},
}

\bib{eisenstein}{article}{
      author={Eisenstein, G.},
       title={Aufgaben und {L}ehrs\"{a}tze},
        date={1844},
        ISSN={0075-4102},
     journal={J. Reine Angew. Math.},
      volume={27},
       pages={281\ndash 284},
         url={https://doi.org/10.1515/crll.1844.27.281},
      review={\MR{1578398}},
}

\bib{nelvis}{article}{
      author={Fornasin, Nelvis},
       title={$\eta$ invariants under degeneration to cone-edge singularities},
     journal={Doctoral Dissertation, Albert-Ludwigs-Universit\"{a}t Freiburg},
}

\bib{friedrich}{article}{
      author={Friedrich, Th.},
       title={Zur {A}bh\"{a}ngigkeit des {D}irac-operators von der
  {S}pin-{S}truktur},
        date={1984},
        ISSN={0010-1354},
     journal={Colloq. Math.},
      volume={48},
      number={1},
       pages={57\ndash 62},
         url={https://doi.org/10.4064/cm-48-1-57-62},
      review={\MR{750754}},
}

\bib{hitchin}{article}{
      author={Hitchin, Nigel},
       title={Harmonic spinors},
        date={1974},
        ISSN={0001-8708},
     journal={Advances in Math.},
      volume={14},
       pages={1\ndash 55},
         url={https://doi.org/10.1016/0001-8708(74)90021-8},
      review={\MR{0358873}},
}

\bib{jk}{article}{
      author={Joyce, Dominic},
      author={Karigiannis, Spiro},
       title={A new construction of compact torsion-free {$G_2$}-manifolds by
  gluing families of {E}guchi–{H}anson spaces},
        date={2018},
        note={https://arxiv.org/pdf/1707.09325.pdf},
}

\bib{joyce-book}{book}{
      author={Joyce, Dominic~D.},
       title={Compact manifolds with special holonomy},
      series={Oxford Mathematical Monographs},
   publisher={Oxford University Press, Oxford},
        date={2000},
        ISBN={0-19-850601-5},
      review={\MR{1787733}},
}

\bib{joyce-g2}{article}{
      author={Joyce, Dominic~D.},
       title={Compact {R}iemannian {$7$}-manifolds with holonomy {$G_2$}. {I},
  {II}},
        date={1996},
        ISSN={0022-040X},
     journal={J. Differential Geom.},
      volume={43},
      number={2},
       pages={291\ndash 328, 329\ndash 375},
         url={http://projecteuclid.org/euclid.jdg/1214458109},
      review={\MR{1424428}},
}

\bib{kirklesch}{article}{
      author={Kirk, Paul},
      author={Lesch, Matthias},
       title={The {$\eta$}-invariant, {M}aslov index, and spectral flow for
  {D}irac-type operators on manifolds with boundary},
        date={2004},
        ISSN={0933-7741},
     journal={Forum Math.},
      volume={16},
      number={4},
       pages={553\ndash 629},
         url={https://doi.org/10.1515/form.2004.027},
      review={\MR{2044028}},
}

\bib{kovalev}{article}{
      author={Kovalev, Alexei},
       title={Twisted connected sums and special {R}iemannian holonomy},
        date={2003},
        ISSN={0075-4102},
     journal={J. Reine Angew. Math.},
      volume={565},
       pages={125\ndash 160},
         url={https://doi.org/10.1515/crll.2003.097},
      review={\MR{2024648}},
}

\bib{mp-dirac}{article}{
      author={Miatello, Roberto~J.},
      author={Podest\'{a}, Ricardo~A.},
       title={The spectrum of twisted {D}irac operators on compact flat
  manifolds},
        date={2006},
        ISSN={0002-9947},
     journal={Trans. Amer. Math. Soc.},
      volume={358},
      number={10},
       pages={4569\ndash 4603},
         url={https://doi.org/10.1090/S0002-9947-06-03873-6},
      review={\MR{2231389}},
}

\bib{nordstrom}{article}{
      author={Nordstr\"om, Johannes},
       title={Extra-twisted connected sum {$G_2$}-manifolds},
        date={2018},
        note={https://arxiv.org/abs/1809.09083},
}

\bib{ws}{inproceedings}{
      author={Salamon, Dietmar~A.},
      author={Walpuski, Thomas},
       title={Notes on the octonions},
        date={2017},
   booktitle={Proceedings of the {G}\"{o}kova {G}eometry-{T}opology
  {C}onference 2016},
   publisher={G\"{o}kova Geometry/Topology Conference (GGT), G\"{o}kova},
       pages={1\ndash 85},
      review={\MR{3676083}},
}

\end{biblist}
\end{bibdiv}
\bibliographystyle{alpha}

\end{document}